\documentclass{elsarticle}

\usepackage{amsfonts,amsmath,amssymb,amscd,subfigure}
\usepackage[usenames]{color}
\usepackage{multicol}
\usepackage{algorithmic}
\usepackage{algorithm}
\usepackage{graphicx}
\usepackage{bm}
\usepackage{psfrag}

\newtheorem {prop}{Proposition}
\newtheorem {rem}{Remark}
\newtheorem{theorem}{Theorem}
\newtheorem{lem}{Lemma}

\newtheorem{defn}{Definition}
\newproof{proof}{Proof}

\numberwithin{table}{section}
\numberwithin{equation}{section}
\numberwithin{figure}{section}

\newcommand{\BA}[1]{\textcolor{blue}{[BA: #1]}}

\newcommand{\REAL}{\ensuremath{\mathbb{R}}}

\providecommand{\norm}[1]{\lVert#1\rVert}
\newcommand{\comment}[1]{} 

\newcommand{\bdC}{\mathbf{C}}
\newcommand{\bdu}{\mathbf{u}}
\newcommand{\bdx}{\mathbf{x}}
\newcommand{\bdxp}{\mathbf{x^\prime}}

\newcommand{\bdb}{\mathbf{b}}
\newcommand{\bdv}{\mathbf{v}}
\newcommand{\bdzero}{\mathbf{0}}

\newcommand{\bdxi}{\ensuremath{\mbox{\boldmath$\xi$}}}
\newcommand{\OmegaNLC}{\overline{\overline{\Omega}}} 
\newcommand{\OmegaNLCi}{\overline{\overline{\Omega^{(i)}}}}

\newcommand{\horizon}{\mathcal{H}_\bdx}

\newcommand{\myChi}{\chi_\delta(\bdx - \bdxp)}

\journal{Applied Mathematics and Computation}

\begin{document}

\begin{frontmatter}

    \title{Variational Theory and Domain Decomposition for Nonlocal Problems}

    \author[tobb,lsu]{Burak Aksoylu\fnref{fn1}}
    \ead{baksoylu@etu.edu.tr}

    \author[sandia]{Michael L. Parks\corref{cor1}\fnref{fn2}}
    \ead{mlparks@sandia.gov}

    \cortext[cor1]{Corresponding author}

    \fntext[fn1]{B. Aksoylu was supported in part by NSF DMS-1016190 and his visits to Sandia National Laboratories were partially supported by NSF LA EPSCoR and Louisiana Board of Regents LINK program.}

    \fntext[fn2]{Sandia National Laboratories is a multi-program laboratory operated by Sandia Corporation, a wholly owned subsidiary of Lockheed Martin company, for the U.S. Department of Energy's National Nuclear Security Administration under contract DE-AC04-94AL85000.}

    \address[tobb]{TOBB University of Economics and Technology, Department of Mathematics, Ankara, 06560, Turkey}
    \address[lsu]{Louisiana State University, Department of Mathematics, Baton Rouge, LA 70803-4918 USA}
    \address[sandia]{Sandia National Laboratories, Applied Mathematics and Applications, P.O. Box 5800, MS 1320, Albuquerque, NM 87185-1320 USA}

\comment{
\author{Burak Aksoylu\affil{1}$^,$\affil{2}\comma\corrauth\comma\footnotemark[2], Michael L. Parks\affil{3}\comma\footnotemark[3]}

\address{
\affilnum{1}\ TOBB University of Economics and Technology, Department of Mathematics, Ankara, 06560, Turkey.\\
\affilnum{2}\ Department of Mathematics, Louisiana State University, Baton Rouge, LA 70803-4918 USA.\\
\affilnum{3}\ Applied Mathematics and Applications Department, Sandia National Laboratories, P.O. Box 5800, MS 1320, Albuquerque, NM 87185-1320 USA.
}

\footnotetext[2]{E-mail: baksoylu@etu.edu.tr}

\footnotetext[3]{Sandia National Laboratories is a multi-program laboratory operated by Sandia Corporation, a wholly owned subsidiary of Lockheed Martin company, for the U.S. Department of Energy's National Nuclear Security Administration under contract DE-AC04-94AL85000.}

\corraddr{Burak Aksoylu, TOBB University of Economics and Technology,
Department of Mathematics, Ankara, 06560, Turkey.}

\cgs{B. Aksoylu was supported in part by NSF DMS-1016190 and his visits to Sandia National Laboratories were partially supported by NSF LA EPSCoR and Louisiana Board of Regents LINK program.}

\received{27 August 2010}
\revised{}
\noaccepted{}
}

\begin{abstract}
In this article we present the first results on domain decomposition methods for nonlocal operators. We present a nonlocal variational formulation for these operators and establish the well-posedness of associated boundary value problems, proving a nonlocal Poincar\'{e} inequality. To determine the conditioning of the discretized operator, we prove a spectral equivalence which leads to a mesh size independent upper bound for the condition number of the stiffness matrix.  We then introduce a nonlocal two-domain variational formulation utilizing nonlocal transmission conditions, and prove equivalence with the single-domain formulation. A nonlocal Schur complement is introduced. We establish condition number bounds for the nonlocal stiffness and Schur complement matrices. Supporting numerical experiments demonstrating the conditioning of the nonlocal one- and two-domain problems are presented.
\end{abstract}

\begin{keyword}
Domain decomposition \sep nonlocal substructuring \sep nonlocal operators \sep nonlocal Poincar\'{e} inequality \sep p-Laplacian \sep peridynamics \sep nonlocal Schur complement \sep condition number.
\end{keyword}

\end{frontmatter}


\comment{
\section*{Big Picture}
\BA{This is a summary of the nonlocal theory we provide in this paper.
You can rewrite the intro and other parts under the light of
the below summary.\\
First, we provide a variational theory for nonlocal problems.  With
this theory, for PD realistic kernel functions, we prove the
well-posedness of the weak formulation of nonlocal BVPs with Dirichlet
BC.  In addition, we prove a spectral equivalence estimate which leads
to a mesh size independent upper bound for the condition number of the
stiffness matrix. \\
Secondly, we construct a nonlocal domain decomposition
framework in which prove the equivalence of single domain and
two-domain decompositions. We also prove that the condition number
of the resulting Schur complement matrix is no worse than the that
of the stiffness matrix.\\
Thirdly, we support our conditioning theory with numerical experiments.
}
}

\section{Introduction} \label{sec:Introduction}

Domain decomposition methods where the subdomains do not overlap are called \emph{substructuring methods}, reflecting their origins and long use within the structural analysis community \cite{Przemieniecki:1963:Substructure}. These methods solve for unknowns only along the interface between subdomains, thus decoupling these domains from each other and allowing each subdomain to then be solved independently. One may solve for the primal field variable on the interface, generating a Dirichlet boundary value problem on each subdomain (these are Schur complement methods, see \cite{Smith:1996:DD} and references cited therein), or solve for the dual field variable on the interface, generating a Neumann boundary value problem on each subdomain (these are dual Schur complement methods, see \cite{Farhat:1991:FETI,Farhat:1994:CMA,Farhat:2000:SecondGenerationFETI,Rixen:1999:HeterogeneousScaling}). Hybrid dual-primal methods have also been developed \cite{Farhat:2001:FETI-DP}.

As domain decomposition methods are frequently employed on massively parallel computers, only \emph{scalable} methods are of interest, meaning that the condition number of the interface problem does not grow (or, only grows weakly) with the number of subdomains. Scalable or weakly scalable methods are generated by application of an appropriate preconditioner to the interface problem. This preconditioner requires the solution of a coarse problem to propagate error globally; see any of the references \cite{Bramble:1986:BPS,Bramble:1987:BPS,Bramble:1988:BPS,Bramble:1989:BPS,Farhat:1994:OptimalConvergenceDirichlet,Mandel:1993:BDD,Klawonn:2001:FETI,Dohrmann:2003:BDDC}. For a general overview of domain decomposition, the reader is directed to the excellent texts \cite{Smith:1996:DD,Quarteroni:1999:DDforPDEs,Toselli:2005:DDBook}.

All of the methods referenced above have in common that they are domain decomposition approaches for \emph{local} problems. In this article, we propose and study a domain decomposition method for the \emph{nonlocal} Dirichlet boundary value problem
\begin{align} \label{eq:1DomainPDstrong}
 \mathcal{L}(\bdu) = \bdb(\bdx), \qquad \bdx \in \Omega,
\end{align}
where
\begin{align} \label{eq:strongOp}
 \mathcal{L}(\bdu) := -\int_{\Omega \cup \mathcal{B}\Omega} \bdC (\bdx,\bdxp) \, [\bdu(\bdxp) - \bdu(\bdx)] \, d \bdxp.
\end{align}
Let $n$ and $d$ denote the dimensions of the function space and
the spatial domain, respectively. $\Omega \subset \REAL^d$ is a bounded
domain, $\mathcal{B}\Omega$ is given in \eqref{eq:BoundaryOmega},
$\bdb$ is given, and $\bdu(\bdx) \in \REAL^n$ is prescribed for $\bdx
\in \REAL^d \backslash \Omega$.  We prescribe the value of
$\bdu(\bdx)$ outside $\Omega$ and not just on the boundary of
$\Omega$, owing to the nonlocal nature of the problem.

Nonlocal models are useful where classical (local) models cease to be
predictive. Examples include porous media flow
\cite{Cushman:1993:NonlocalReactiveFlow,Dagan:1994:NonlocalReactiveFlow,sinhaEwingLazarov:2006:parabolicIntegroDE},
turbulence \cite{Bakunin:2008:Turblence}, fracture of solids, stress
fields at dislocation cores and cracks tips, singularities present at
the point of application of concentrated loads (forces, couples, heat,
etc.), failure in the prediction of short wavelength behavior of
elastic waves, microscale heat transfer, and fluid flow in microscale
channels~\cite{Eringen:2002:NonlocalBook}. These are also cases where
microscale fields are nonsmooth. Consequently, nonlocal models are
also useful for multiscale modeling. Recent examples of nonlocal
multiscale modeling include the upscaling of molecular dynamics to
nonlocal continuum mechanics \cite{Seleson:2009:UpscalingMDtoPD}, and
development of a rigorous multiscale method for the analysis of
fiber-reinforced composites capable of resolving dynamics at
structural length scales as well as the length scales of the
reinforcing fibers~\cite{Alali:2009:HeterogeneousPD}. Progress towards
a nonlocal calculus is reported in
\cite{Gunzburger:2009:NonlocalCalculus}. Development and analysis of a
nonlocal diffusion equation is reported in
\cite{Andreu:2008:NeumannPrb,Andreu:2008:NeumannPrb2,Andreu:2009:pLaplace}.
Theoretical developments for general class of integro-differential
equation related to the fractional Laplacian are presented in
\cite{caffarelliSilvestre2007,caffarelliSilvestre2009,silvestre2006}. Mathematical and numerical analysis for linear nonlocal peridynamic boundary problems appears in \cite{Du:2010:MathAnalPD,Zhou:2010:MathAnalPD}.
We discuss in \S\ref{sec:Interpretations} some specific contexts where
the nonlocal operator $\mathcal{L}$ appears, and the assumptions
placed upon $\mathcal{L}$ by those interpretations.

To the best of authors' knowledge, this article represents the first
work on domain decomposition methods for nonlocal models. Our aim is
to generalize iterative substructuring methods to a nonlocal setting
and characterize the impact of nonlocality upon the scalability of
these methods. To begin our analysis, we first develop a weak form for
\eqref{eq:1DomainPDstrong} in \S\ref{sec:WeakForm}.  The main
theoretical construction for conditioning is in
\S\ref{sec:condNumberTheory}. We establish spectral equivalences to
bound the condition numbers of the stiffness and Schur complement
matrices.  For that, we prove a nonlocal Poincar\'{e} inequality for
the lower bound and a dimension dependent estimate for the upper
bound.  This leads to the novel result that \emph{the condition number
  of the discrete nonlocal operator can be bounded independently of the
  mesh size}.  In \S\ref{sec:DDPD}, we construct a suitable nonlocal
domain decomposition framework with special attention to transmission
conditions. Then, we prove the equivalence of the boundary value problems
corresponding to the single domain and the two-domain
decomposition. In \S\ref{sec:NonlocalSubstructure}, we first define a
\emph{discrete energy minimizing extension}, a nonlocal analog of
discrete harmonic extension in the local case, to study the
conditioning of the Schur complement in the nonlocal setting. We
discretize our two-domain weak form to arrive at a \emph{nonlocal
  Schur complement}. We perform numerical studies to validate our
theoretical results.  Finally in \S\ref{sec:conclusions}, we draw
conclusions about conditioning and suggest future research directions
for nonlocal domain decomposition methods.


\section{Interpretations of the Operator  
$\mathcal{L}$} \label{sec:Interpretations}

The operator $\mathcal{L}$ appears in many different application areas, from evolution equations for species population densities \cite{Carrillo:2005:PopulationModels} to image processing \cite{Gilboa:2008:NonlocalImage}. We review two specific contexts in which the operator $\mathcal{L}$ of \eqref{eq:1DomainPDstrong} is utilized, paying special attention to associated assumptions these interpretations place upon $\bdC$ in $\mathcal{L}$. In all cases, we find $\bdC$ to have local support about $\bdx$, meaning that we must prescribe Dirichlet boundary conditions only for
\begin{align} \label{eq:BoundaryOmega}
 \mathcal{B}\Omega := \textrm{supp}(\bdC) \backslash \Omega,
\end{align}
as depicted in Figure \ref{fig:1Domain}. Furthermore, throughout this
article we assume an integrable $\bdC$.

\begin{figure}
 \begin{center}
\psfrag{1}{$\Omega$} \psfrag{2}{$\mathcal{B} \Omega$}
 \scalebox{0.75}{\includegraphics{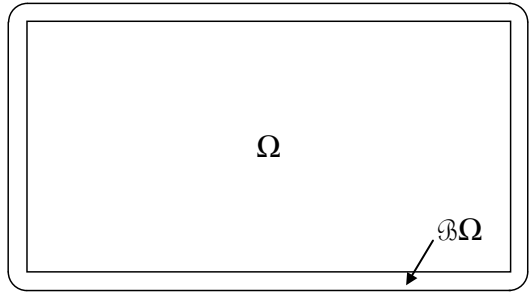}}
 \end{center}
 \caption{Typical domain for \eqref{eq:1DomainPDstrong}. $\bdu$ is prescribed in $\mathcal{B}\Omega$, and we solve for $\bdu$ in $\Omega$.}
 \label{fig:1Domain}
\end{figure}

\subsection{Nonlocal Diffusion Processes}
The equation
\begin{align} \label{eq:nonlocalDiffusion}
 u_t(\bdx,t) = \mathcal{L}(u(\bdx,t))
\end{align}
is an instance of a nonlocal p-Laplace equation for $p=2$, and has been used to model nonlocal diffusion processes, see \cite{Birch:2010:NonlocalDiffusion}, \cite{Andreu:2009:pLaplace} and the references cited therein. In this setting, $u(\bdx,t) \in \REAL$ is the density at the point $\bdx$ at time $t$ of some material, and we assume $C(\bdx,\bdxp) = C(\bdx-\bdxp)$ is translation invariant. Then, $\int_{\REAL^d} C(\bdxp-\bdx) u(\bdxp,t) d \bdxp$ is the rate at which material is arriving at $\bdx$ from all other points in $\textrm{supp}(C)$, and $-\int_{\REAL^d} C(\bdxp-\bdx) u(\bdx,t) d \bdxp$ is the rate at which material departs $\bdx$ for all other points in $\textrm{supp}(C)$~\cite{Fife:2003:NonclassicalParabolic,Andreu:2009:pLaplace}.

In this interpretation of \eqref{eq:1DomainPDstrong} the following restrictions are placed upon $C$ in $\mathcal{L}$. It is assumed that $C: \REAL^d \rightarrow \REAL$ is a nonnegative, radial, continuous function that is strictly positive in a ball of radius $\delta$ about $\bdx$ and zero elsewhere. Additionally, it is assumed that $\int_\Omega C(\bdxi) d \bdxi~<~\infty$.

\subsection{Nonlocal Solid Mechanics}
The equation
\begin{align} \label{eq:PDEOM}
 \bdu_{tt}(\bdx,t) = \mathcal{L}(\bdu(\bdx,t)) + \bdb(\bdx)
\end{align}
is the linearized peridynamic equation \cite[eqn. (56)]{Silling:2000:PD}. The corresponding time-independent (``peristatic'') equilibrium equation is \eqref{eq:1DomainPDstrong}. Peridynamics is a nonlocal reformulation of continuum mechanics that is oriented toward deformations with discontinuities, see \cite{Silling:2000:PD,Silling:2005:Meshfree,Silling:2007:PDStates} and the references therein. In this context, $\bdu \in \REAL^n$ is the displacement field for the body $\Omega$, and $\bdC(\bdx,\bdxp)$ is a stiffness tensor, also known as a \emph{micromodulus} tensor.

In this interpretation of \eqref{eq:1DomainPDstrong} the following
restrictions are placed upon $\bdC$ in $\mathcal{L}$. It is assumed
that $\bdC$ is integrable and strictly positive definite in the
\emph{neighborhood} of $\bdx$, $\horizon$, defined as
\begin{equation} \label{defn:horizon}
\horizon := \{ \bdxp \in \REAL^d \; : \; \norm{\bdxp -  \bdx} \leq \delta \},
\end{equation}
where $\delta > 0$ is called the \emph{horizon}. These assumptions are made because they are sufficient to ensure material
stability~\cite[pp. 191-194]{Silling:2000:PD}. It is also assumed that $\bdC =
\bdzero$ for $\norm{\bdxp - \bdx} > \delta$. If the material is
elastic, it follows that $\bdC(\bdx,\bdxp)$ is symmetric (e.g.,
$\bdC(\bdx,\bdxp)^T = \bdC(\bdx,\bdxp)$). Further, it is assumed that
$\bdC$ is symmetric with respect to its arguments (e.g.,
$\bdC(\bdx,\bdxp) = \bdC(\bdxp,\bdx)$). This follows from imposing
that the integrand of \eqref{eq:strongOp} must be anti-symmetric in its
arguments, e.g.,
\begin{align*}
\bdC(\bdx,\bdxp) \, [\bdu(\bdxp) - \bdu(\bdx)] = - \bdC(\bdxp,\bdx) \, [\bdu(\bdx) - \bdu(\bdxp)]
\end{align*}
in accordance with Newton's third law.

\section{A Nonlocal Variational Formulation} \label{sec:WeakForm}

Here we present a variational formulation of the nonlocal equation
\eqref{eq:1DomainPDstrong}. For peridynamics, this was presented by
Emmerich and Weckner in \cite{Emmrich:2007:DiscretizePD}. An analogous
expression also appears in \cite[eqn. (75)]{Silling:2000:PD}, as well
as \cite{Gunzburger:2009:NonlocalCalculus}.

Our construction takes place on the domain under consideration and
its nonlocal boundary, i.e., $\Omega \cup \mathcal{B} \Omega$.
We define the \emph{nonlocal closure} of $\Omega$ as follows:
\begin{equation*}
\OmegaNLC:= \Omega \cup \mathcal{B} \Omega.
\end{equation*}
We will utilize the function space
\begin{equation} \label{eq:pdFuncSpace}
 V := L_{2,0}^n(\OmegaNLC) = \left\{ \bdv \in L_2^n(\OmegaNLC):
 \bdv\vert_{\mathcal{B} \Omega} = \bdzero \right\},
\end{equation}
and the inner product
\begin{equation*}
(\bdu,\bdv) := \int_{\OmegaNLC} \bdu \, \bdv~d\bdx.
\end{equation*}
The weak formulation of \eqref{eq:1DomainPDstrong} is the following:
Given $\bdb(\bdx) \in L_{2}^n(\Omega)$, find $\bdu(\bdx) \in V$ such that
\begin{alignat}{2} \label{eq:pdEquMotionWeak}
a(\bdu,\bdv) = (\bdb,\bdv) \qquad \forall \bdv \in V,
\end{alignat}
where
\begin{equation} \label{eq:bilinearFormA}
a(\bdu,\bdv) := - \int_{\OmegaNLC}
\left\{ \int_{\OmegaNLC}
\bdC(\bdx,\bdxp) \, [\bdu(\bdxp) -
   \bdu(\bdx)]~d\bdxp \right\} \, \bdv(\bdx)~d\bdx.
\end{equation}

We assume that the iterated integral in \eqref{eq:bilinearFormA} is finite:
\begin{equation*}
 -\int_{\OmegaNLC} \left\{ \int_{\OmegaNLC} \bdC(\bdx,\bdxp) \,
[\bdu(\bdxp) - \bdu(\bdx)]~d\bdxp \right\} \, \bdv(\bdx)~d\bdx < \infty,
\end{equation*}
and that $\bdC(\bdx,\bdxp) \, [\bdu(\bdxp) - \bdu(\bdx)]$ is anti-symmetric in its arguments. Combining these observations with Fubini's Theorem gives the identity
\begin{align} \label{eq:equivBilinearForms}
& - \int_{\OmegaNLC}
\left\{ \int_{\OmegaNLC}
\bdC(\bdx,\bdxp) \, [\bdu(\bdxp) - \bdu(\bdx)]~d\bdxp \right\}
\, \bdv(\bdx)~d\bdx = \\
\nonumber &\frac{1}{2} \int_{\OmegaNLC}
\left\{ \int_{\OmegaNLC} \bdC(\bdx,\bdxp) \,
[\bdu(\bdxp) - \bdu(\bdx)] [\bdv(\bdxp) - \bdv(\bdx)]~d\bdx^\prime \right\}
\,~d\bdx.
\end{align}


For the proof of well-posedness of the nonlocal BVP
\eqref{eq:pdEquMotionWeak}, we utilize the equivalent expression
in \eqref{eq:equivBilinearForms} which induces the following bilinear
form:
\begin{equation} \label{eq:bilinearFormCFormat}
a(\bdu,\bdv) = \frac{1}{2} \int_{\OmegaNLC}
\int_{\OmegaNLC}
\bdC(\bdx,\bdxp) \, [\bdu(\bdxp) - \bdu(\bdx)] \,
[\bdv(\bdxp) - \bdv(\bdx)] ~d\bdxp~d\bdx,
\end{equation}
In \S\ref{sec:NonlocalPoincare}, we will establish the coercivity of
$a(u,u)$ in $V$ in the case of scalar functions, i.e., by setting
$n=1$ in \eqref{eq:pdFuncSpace}.  The continuity of $a(u,v)$ in
$L_{2}(\OmegaNLC)$ follows from \eqref{bilinearFormBdd}.  Furthermore,
$R(v) := (b,v)$ is a bounded linear functional on
$L_2(\OmegaNLC)$. Therefore, well-posedness of
\eqref{eq:pdEquMotionWeak} follows from the Lax-Milgram Lemma; also
see \cite[Sec. 6]{Gunzburger:2009:NonlocalCalculus}.

In 1D with $\OmegaNLC:= [-\delta, 1+\delta]$,
the weak form \eqref{eq:bilinearFormCFormat} becomes
\begin{align} \label{changeOfVariable}
 a(u,u) = \frac{1}{2 }\int_{[-\delta, 1+\delta]}
\int_{[x-\delta, x+\delta] \cap [-\delta, 1+\delta]} C(x,x^\prime)
\left( u(x^\prime) - u(x) \right)^2 dx^\prime dx,
\end{align}
where the limits of integration have been adjusted to account for the
support of $C(x,x^\prime)$, which is assumed to vanish if $\norm{x-x^\prime} > \delta$.  For this problem, the two-dimensional
domain of integration is the parallelogram shown in Figure
\ref{fig:IntegrationDomain}.
For 2D and 3D problems, the domains of
integration are four and six dimensional, respectively.
\begin{figure}
 \begin{center}
 \scalebox{0.6}{\includegraphics{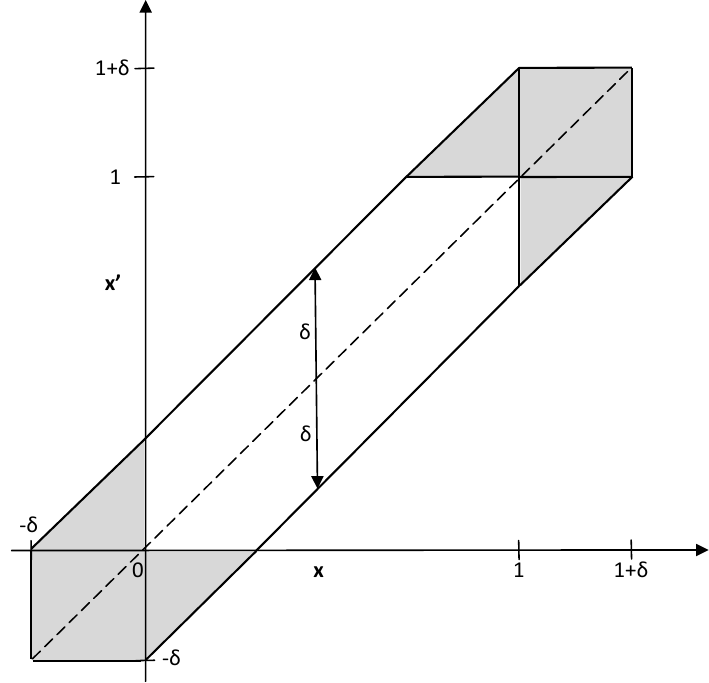}}
 \end{center}
 \caption{Domain of integration for a 1D problem where $\Omega = [0,1]$ and $\mathcal{B} \Omega = [-\delta,0] \cup [1,1+\delta]$. A nonlocal Dirichlet boundary condition is prescribed over $\mathcal{B} \Omega$. The grey region indicates the portion of the integration domain where either or both of $x$, $x^\prime$ lie outside $\Omega$.} \label{fig:IntegrationDomain}
\end{figure}

\section{Nonlocal Spectral Equivalence} \label{sec:condNumberTheory}

The principle result of this section is Theorem \ref{thm:spec_equiv}, a condition number bound for the stiffness matrix arising from a finite element discretization of \eqref{eq:bilinearFormA}. We investigate the conditioning because it determines both the accuracy of the computed numerical solution, as well as the computational effort required by an iterative linear solver to produce the numerical solution. Quantifying the condition number bound is a necessary first step towards developing scalable preconditioners and optimal solvers for nonlocal models.

In the local setting, the classical condition number estimates rely on
a Poincar\'{e} inequality and an inverse inequality for the lower and
upper bound, respectively.  Similarly to the local case, we develop a
nonlocal Poincar\'{e} inequality to be used in the lower bound.  We
prove a nonlocal Poincar\'{e} inequality which is used to establish
the coercivity of the underlying bilinear form. However, for condition
number analysis, one needs a more refined Poincar\'{e} inequality
which involves an explicit $\delta$-quantification.  Such refined
inequality requires substantially more involved analysis, which has
been accomplished by the first author in the companion article
\cite{Aksoylu:2010:NLBVP}.

The $\delta$-quantification is an essential feature in the nonlocal
setting because the lower bound turns out to be dimension dependent,
unlike in the local case.  This dimensional dependence is induced by
the neighborhood $\mathcal{H}_x$ (see \eqref{defn:horizon}), which is
$d$-dimensional in the nonlocal setting but zero-dimensional (a point)
in the local setting.  Dimension dependence in the Poincar\'{e}
inequality is captured by $\delta^m$ (see
\S\ref{sec:NonlocalPoincare}) where the power $m$ exhibits a
dimensional dependence (i.e., $m=m(d)$).

For the upper bound, we prove a direct estimate instead of an inverse
inequality.  Neither the upper bound estimate nor the Poincar\'{e}
inequality requires discrete spaces. Hence, our estimate is valid in infinite
dimensional function spaces, a stronger result than that for the local
setting.

We investigate the effect of the horizon size $\delta$ on the
conditioning of the underlying operators.  Therefore, we reduce the
analysis to the case $\bdC(\bdx,\bdxp) = \chi_\delta(\bdx - \bdxp)$
where $\chi_\delta(\bdx - \bdxp)$ denotes the \emph{canonical} kernel
function whose only role is the representation of the neighborhood in
\eqref{defn:horizon} by a characteristic function.
Namely,
\begin{align} \label{eq:pdKernel}
\chi_\delta(\bdx - \bdxp) := \left\{
\begin{array}{cl}
1, & \norm{\bdx-\bdx^\prime} \leq \delta \\
0, & \textrm{otherwise.}
\end{array}
\right.
\end{align}
Note that $\chi_\delta(\bdx - \bdxp)$ is a radial function, which describes isotropic materials~\cite{Silling:2007:PDStates}.  For the
remainder of this article, we will restrict our discussion to scalar
problems. Namely, we set $n=1$ in \eqref{eq:pdFuncSpace} which yields,
for instance, $\bdu(\bdx) = u(\bdx)$, $\bdC(\bdx,\bdxp) = C(\bdx,\bdxp)$, etc.
Therefore, the bilinear form under consideration becomes
\begin{equation} \label{eq:canonicalBilinearForm}
a(u,v) = \frac{1}{2} \int_{\OmegaNLC}
\int_{\OmegaNLC} \chi_\delta(\bdx - \bdxp)
[u(\bdxp) - u(\bdx)] [v(\bdxp) - v(\bdx)]~d\bdxp~d\bdx.
\end{equation}
We state an important property of the canonical kernel which
will be used in the upcoming proofs:
\begin{equation} \label{canonicalKernelIntegral}
\int_{\OmegaNLC}\chi_\delta(\bdx - \bdxp)~d\bdx^\prime
\leq w_d \delta^d \quad \bdx \in \OmegaNLC,
\end{equation}
where $w_d$ is the volume of the unit ball in $\REAL^d$.
Note that the equality in \eqref{canonicalKernelIntegral} is
attained when the neighborhood of $\bdx$, $\horizon$ in
\eqref{defn:horizon}, is entirely contained in $\OmegaNLC$, i.e.,
when $\bdx \in \Omega$.

\subsection{Nonlocal Poincar\'{e} inequality} \label{sec:NonlocalPoincare}

In order to establish the coercivity of $a(\cdot,\cdot)$, we prove a
nonlocal Poincar\'{e} inequality.
\begin{prop}
Let $\Omega \subset \REAL^d$ be a bounded domain and
$u \in L_{2,0}(\OmegaNLC)$.
Then, there exists
$\lambda_{Pncr} = \lambda_{Pncr}(\OmegaNLC,\delta) > 0$ such that
\begin{equation} \label{eq:PoincareIneqGeneral}
\lambda_{Pncr} \, \|u\|^2_{L_2(\OmegaNLC)} \leq a(u,u)
\end{equation}
\end{prop}

\begin{proof}
The proof is an extension of the one given in
\cite[Prop. 2.5]{Andreu:2009:pLaplace} for a similar bilinear form.
We construct a finite covering for $\OmegaNLC$ using
strips of width $\delta / 2$ as follows:
\begin{eqnarray}
S_{-1} & := & \{ \bdx \in \mathcal{B} \Omega :
 \frac{\delta}{2} \leq \mathrm{dist}(\bdx,\partial \Omega)
\leq \delta \},\\
S_0 & := & \{ \bdx \in \mathcal{B} \Omega \setminus S_{-1}: \mathrm{dist}(\bdx,S_{-1}) \leq \frac{\delta}{2} \},\\
S_1 & := & \{ \bdx \in \Omega: \mathrm{dist}(\bdx,\partial \Omega)
\leq \frac{\delta}{2} \},\\
S_j & := & \{ \bdx \in \Omega \setminus \bigcup_{k=1}^{j-1} S_k:
\mathrm{dist}(\bdx,S_{j-1}) \leq \frac{\delta}{2} \}, \quad j=1 \ldots, l,
\end{eqnarray}
where $\mathrm{dist}$ denotes the shortest distance in the usual
Euclidean sense.  The number of strips covering $\Omega$ is
$l = l(\OmegaNLC,\delta)$.

We trivially have the following for $j=0, \ldots, l$:
\begin{equation*}
\int_{\OmegaNLC} \int_{\OmegaNLC} \myChi \,
|u(\bdx^\prime) - u(\bdx)|^2~d\bdx^\prime d\bdx \geq
\int_{S_j} \int_{S_{j-1}} \myChi \,
|u(\bdx^\prime) - u(\bdx)|^2~d\bdx^\prime d\bdx.
\end{equation*}

Using
$
|u(\bdx)|^2 = |u(\bdxp) - \{ u(\bdxp) - u(\bdx) \}|^2 \leq
2 \{ |u(\bdxp) - u(\bdx)|^2 + |u(\bdxp)|^2 \},
$
a change in the order of integration, and the following result (obtained
from \eqref{canonicalKernelIntegral})
\begin{equation*}
\int_{S_{j}} \chi_\delta(\bdx - \bdxp)~d\bdx^\prime
\leq w_d \delta^d,
\end{equation*}
we obtain the following:
\begin{eqnarray*}
&& \int_{S_j} \int_{S_{j-1}} \chi_\delta(\bdx - \bdxp) \, |u(\bdx^\prime) - u(\bdx)|^2~d\bdx^\prime d\bdx \\
& \geq &   \frac{1}{2} \int_{S_j} \int_{S_{j-1}} \chi_\delta(\bdx - \bdxp) \, |u(\bdx)|^2~d\bdxp d\bdx
- \int_{S_j} \int_{S_{j-1}} \chi_\delta(\bdx - \bdxp) \, |u(\bdx^\prime)|^2~d\bdxp d\bdx\\
& = & \frac{1}{2} \int_{S_j} \left\{ \int_{S_{j-1}} \chi_\delta(\bdx - \bdxp)~d\bdxp \right\} \, |u(\bdx)|^2~d\bdx
- \int_{S_{j-1}} \left\{ \int_{S_j} \chi_\delta(\bdx - \bdxp) d\bdx \right\} \, |u(\bdxp)|^2~d\bdxp \\
& \geq & \frac{1}{2} \int_{S_j} \left\{ \int_{S_{j-1}} \chi_\delta(\bdx - \bdxp)~d\bdxp \right\} \, |u(\bdx)|^2~d\bdx
- w_d \, \delta^d \int_{S_{j-1}} |u(\bdx^\prime)|^2~d\bdxp \\
& \geq & \frac{1}{2} \, \min_{x \in \overline{S}_j} \int_{S_{j-1}} \chi_\delta(\bdx - \bdxp) ~d\bdxp \int_{S_j} |u(\bdx)|^2~d\bdx
- w_d \, \delta^d \, \int_{S_{j-1}} |u(\bdx^\prime)|^2~d\bdx^\prime.
\end{eqnarray*}
The function
\begin{equation*}
F(\bdx) := \int_{S_{j-1}} \chi_\delta(\bdx - \bdxp)~d\bdxp,
\quad \bdx \in \overline{S_j}
\end{equation*}
is continuous, which follows from continuity of the integral operator and the fact that $\chi_\delta(\bdx - \bdxp)$ is
integrable.
By construction of the covering, we have
\begin{equation*}
S_{j-1} \cap B(\bdx,\delta) \neq \emptyset, \quad \bdx \in S_j,
\end{equation*}
where $B(\bdx,\delta)$ is a ball of radius $\delta$ centered at $\bdx$. Hence, we obtain
\begin{equation*} 
F(\bdx) =
\text{measure}(\{ \bdx \in \overline{S_j}: S_{j-1} \cap B(\bdx,\delta)\})> 0, \quad \bdx \in \overline{S_j}.
\end{equation*}
Therefore by continuity,
$F(\bdx)$ attains its infimum in  $\overline{S_j}$ and we conclude that
\begin{equation*}
\alpha_j := \min_{\bdx \in \overline S_j} F(\bdx) > 0.
\end{equation*}
Consequently, we have the inequality:
\begin{equation} \label{ballStep1}
\frac{\alpha_j}{4} \int_{S_j} |u(\bdx)|^2~d\bdx \leq
a(u,u) + w_d \, \delta^d \int_{S_{j-1}} |u(\bdxp)|^2~d\bdxp.
\end{equation}
From the boundary condition, we get
\begin{equation*}
\int_{S_{-1}}|u(\bdx)|^2~d\bdx = \int_{S_{0}}|u(\bdx)|^2~d\bdx = 0.
\end{equation*}
Moreover due to the boundary condition and \eqref{ballStep1}, we get
\begin{equation} \label{pf:PoincareStep-1}
\frac{\alpha_1}{4} \int_{S_1} |u(\bdx)|^2~d\bdx \leq a(u,u)
\end{equation}
For the cases $j=2,3$, we respectively have:
\begin{eqnarray}
\label{pf:PoincareStep0}
\frac{\alpha_2}{4} \int_{S_2} |u(\bdx)|^2~d\bdx & \leq &
a(u,u) + w_d \, \delta^d \int_{S_1} |u(\bdxp)|^2~d\bdxp \\
\label{pf:PoincareStep1}
\frac{\alpha_3}{4} \int_{S_3} |u(\bdx)|^2~d\bdx & \leq &
a(u,u) + w_d \, \delta^d\int_{S_2} |u(\bdxp)|^2~d\bdxp.
\end{eqnarray}
To relate \eqref{pf:PoincareStep1} to right-hand side of \eqref{pf:PoincareStep0},
multiply \eqref{pf:PoincareStep0} by $(4 \, w_d \, \delta^d) / \alpha_2$:
\begin{equation} \label{pf:PoincareStep2_A}
\frac{\alpha_3}{4} \int_{S_3} |u(\bdx)|^2~d\bdx \leq
(1+ \frac{4 \, w_d \, \delta^d}{\alpha_2})~a(u,u) +
\frac{4 \, (w_d \, \delta^d)^2}{\alpha_2} \int_{S_1} |u(\bdxp)|^2~d\bdxp.
\end{equation}
Then using \eqref{pf:PoincareStep-1}, \eqref{pf:PoincareStep2_A} becomes:
\begin{equation*}
\frac{\alpha_3}{4}~ \int_{S_3} |u(\bdx)|^2~d\bdx \leq
\left(1+ \frac{4 \,  w_d \, \delta^d}{\alpha_2} +
\frac{ (4 \, w_d \, \delta^d)^2}{\alpha_1 \, \alpha_2} \right) ~ a(u,u).
\end{equation*}
Continuing this process, we see the existence of a constant
$c(\OmegaNLC, \delta)$ satisfying:
\begin{equation} \label{ballStep2}
\frac{\alpha_j}{4}~ \int_{S_j} |u(\bdx)|^2~d\bdx \leq c(\OmegaNLC, \delta) ~
a(u,u), \quad j=-1, \ldots, l.
\end{equation}
Adding \eqref{ballStep2} for $j=-1, \ldots, l$ and using the fact
that the covering of $\OmegaNLC$ is composed of disjoint strips, i.e.,
$\OmegaNLC = \cup_{k=-1}^l S_k,~~~ S_j \cap S_k = \emptyset,~j \neq k$,
we arrive at the coercivity result.
\end{proof}

\begin{rem}
  The coercivity proof in \cite[Prop. 2.5]{Andreu:2009:pLaplace}
  assumes a continuous kernel function. The coercivity proof we
  provide can be generalized to any nonnegative locally integrable
  radial kernel function which satisfies  $C(r) > 0$ on $(0, \delta)$ see
  the companion article \cite[Lemma 2.4]{Aksoylu:2010:NLBVP}.
\end{rem}

\begin{rem}
The above Poincar\'{e} type inequality can be established for general Dirichlet boundary
conditions, i.e., $u \in L_2(\OmegaNLC)$ with
$u |_{\mathcal{B} \Omega} \neq 0$. In this case,
the inequality statement reads as follows:
\begin{equation} \label{PoincareGeneralBC}
\lambda_{Pncr}(\OmegaNLC, \delta) \, \|u\|^2_{L_2(\OmegaNLC)}
\leq a(u,u) + \int_{S_{-1}} |u(\bdx)|^2~d\bdx,
\end{equation}
where $S_{-1}$ is the outermost strip of the covering of
$\OmegaNLC$. Deducing a coercivity estimate from \eqref{PoincareGeneralBC}
seems impossible unless a zero nonloncal boundary condition is assumed.
For mixed and Neumann type boundary conditions, see the companion paper
\cite[Remark 2.5]{Aksoylu:2010:NLBVP}.
\end{rem}

\begin{rem}
Coercivity of the bilinear form has also been established in
\cite{Gunzburger:2009:NonlocalCalculus}
under the condition (see \cite[Eq. (6.1)]{Gunzburger:2009:NonlocalCalculus})
that
\begin{equation} \label{GLcoercivityCond}
\int_{\mathcal{B} \Omega} C(\bdx,\bdxp)~d\bdxp \geq c > 0, \quad \bdx \in \Omega.
\end{equation}
This condition is stringent because it assumes that all interior points interact directly with the nonlocal boundary $\mathcal{B}\Omega$, a situation only possible if the horizon $\delta$ is on the order of $\vert \Omega \vert$. For applications of practical interest especially in peridynamics, horizon is set to be $\delta \ll \vert \Omega \vert$ because problems with large $\delta$ are computationally intractable. The coercivity proof given in this article does not assume \eqref{GLcoercivityCond}.
\end{rem}

For the condition number analysis, $\delta$-quantification is essential.
In the companion article \cite{Aksoylu:2010:NLBVP}, the first author
gives a more refined nonlocal Poincar\'{e} inequality. Namely,
for sufficiently small $\delta$:
\begin{equation} \label{poincareDelta}
\lambda_{refined}(\OmegaNLC) \, \delta^{d+2} \|u\|^2_{L_2(\OmegaNLC)}
\leq a(u,u).
\end{equation}
Note that $\lambda_{refined}$ does not depend on $\delta$. In order
to see why $\lambda_{Pncr}(\OmegaNLC,\delta)$ can be refined to a
constant $\lambda_{refined}(\OmegaNLC)$, we proceed with a 1D demonstration.

\subsection{Demonstration of explicit $\delta$-dependence in the nonlocal Poincar\'{e} inequality} \label{sec:choose_m}

We demonstrate the lower bound in \eqref{poincareDelta} by a 1D
example.  After enforcing a sufficient regularity assumption, we
resort to a Taylor series expansion. For that purpose, we assume that
$u \in C^4(\OmegaNLC)$ with homogenous Dirichlet
boundary conditions enforced on the nonlocal boundary $\mathcal{B}\Omega$.
This demonstration is based on the desire to have the nonlocal bilinear
form converge to its corresponding local (classical) bilinear form as
$\delta \rightarrow 0$. For discussions of convergence of other
nonlocal operators to their classical local counterparts, see
\cite{Weckner:2007:PDConverge,Silling:2008:SmallDelta}.

For the sake of clarity, we utilize the equivalent bilinear form below
given in \eqref{eq:equivBilinearForms} so that the effect of the boundary
condition can easily be seen. We accompany this with a change of variable as follows:

\begin{eqnarray*}
 a(u,u) & = & - \int_{\OmegaNLC}
\left\{ \int_{\OmegaNLC \cap [x-\delta, x+\delta]}
[u(x^\prime) - u(x)]~dx^\prime\right\}u(x)~dx \\
& = & - \int_{\Omega}
\left\{ \int_{[x-\delta, x+\delta]}
[u(x^\prime) - u(x)]~dx^\prime\right\}u(x)~dx \\
& = & - \int_{\Omega}
\left\{ \int_{-\delta}^{\delta}
[u(x+ \varepsilon) - u(x)]~d\varepsilon \right\}u(x)~dx.
\end{eqnarray*}

Using the Taylor expansion
\begin{eqnarray*}
u(x + \varepsilon) & = & u(x) + \frac{\varepsilon}{1!} \frac{du}{dx}(x) +
\frac{\varepsilon^2}{2!} \frac{d^2u}{dx^2}(x) +
\frac{\varepsilon^3}{3!} \frac{d^3u}{dx^3}(x) + \mathcal{O}(\varepsilon^4)
\label{eq:uTaylor},
\end{eqnarray*}
the integrand becomes:
\begin{equation*}
[u(\varepsilon+x) - u(x)] u(x) =  \varepsilon \frac{du}{dx}(x) u(x) +
\frac{\varepsilon^2}{2!} \frac{d^2u}{dx^2}(x) u(x)+
\frac{\varepsilon^3}{3!} \frac{d^3u}{dx^3}(x) u(x) + \mathcal{O}(\varepsilon^4).
\end{equation*}
Hence, we arrive at the following expression using $u|_{\partial \Omega} = 0$
(due to $u|_{\mathcal{B} \Omega} = 0$):
\begin{eqnarray*}
 a(u,u) & = & - \int_{\Omega}
\left\{ \frac{\delta^3}{3} \frac{d^2u}{dx^2}(x) u(x) +
\mathcal{O}(\delta^5) \right\}~dx \\
& = & \frac{\delta^3}{3} \int_{\Omega} \frac{du}{dx}(x) \frac{du}{dx}(x)~dx +
\mathcal{O}(\delta^5).
\end{eqnarray*}
Now, denoting the local bilinear form by
\begin{equation*}
\ell(u,u) := |u|_{H^1(\Omega)}^2,
\end{equation*}
we connect the nonlocal and local bilinear forms:
\begin{eqnarray*}
a(u,u) & = & \frac{\delta^3}{3}~ \ell(u,u) + \mathcal{O}(\delta^5).
\end{eqnarray*}
Therefore, the scaled nonlocal bilinear form asymptotically converges
to the local bilinear form:
\begin{equation} \label{eq:nonlocalConvergence}
3~\delta^{-3}~a(u,u) = \ell(u,u) + \mathcal{O}(\delta^2).
\end{equation}
Using the nonlocal Poincar\'{e} inequality \eqref{eq:PoincareIneqGeneral}
and \eqref{eq:nonlocalConvergence}, we have
\begin{equation*}
\lim_{\delta \rightarrow 0 }~3~\lambda_{Pncr}(\OmegaNLC, \delta) ~\delta^{-3}~
\norm{u}_{L_2(\OmegaNLC)}^2
\leq \ell(u,u).
\end{equation*}
Therefore, for the left hand side to remain finite, we have to enforce
that $\lambda_{Pncr}(\OmegaNLC, \delta) = c(\OmegaNLC) \, \delta^m$ with
$m \geq 3$.  We
desire the largest possible lower bound in the nonlocal
Poincar\'{e} inequality. This implies that $m = 3$, which is in
agreement with \eqref{poincareDelta} and is observed numerically in 1D; see the experiments in \S\ref{sec:auu1D}.

\subsection{An upper bound for $a(u,u)$}
We prove the following dimension dependent estimate:

\begin{lem} \label{lem:upperBound}
Let $\OmegaNLC \subset \REAL^d$ be bounded and $u \in L_{2}(\OmegaNLC)$.
Then, there exists
$\overline{\lambda} > 0$ independent of $\OmegaNLC$ and $\delta$ such that
\begin{equation} \label{eq:upperBound}
a(u,u) \leq \overline{\lambda} \; \delta^d \; \norm{u}_{L_2(\OmegaNLC)}^2.
\end{equation}
\end{lem}
\begin{proof}
Using $(u(\bdxp) - u(\bdx))^2 \leq 2 (u(\bdxp)^2 + u(\bdx)^2)$, we get
\begin{equation} \label{bilinearFormExplicit}
a(u,u) \leq
\int_{\OmegaNLC} \int_{\OmegaNLC} \myChi
( u^2(\bdx) + u^2(\bdxp) )~d\bdxp \, d\bdx.
\end{equation}
Furthermore, by a change in the order of integration and the fact that
$\myChi$ is an even function, one gets
\begin{equation} \label{firstTermEquivalence}
\int_{\OmegaNLC} \int_{\OmegaNLC} \myChi u^2(\bdx)~ d\bdxp \,d\bdx
= \int_{\OmegaNLC} \int_{\OmegaNLC} \myChi u^2(\bdxp)~ d\bdxp\,d\bdx.
\end{equation}
Then using \eqref{firstTermEquivalence}, \eqref{bilinearFormExplicit} becomes:
\begin{equation*} 
a(u,u) \leq 2 \int_{\OmegaNLC} \int_{\OmegaNLC} \myChi \, u^2(\bdx)~d\bdxp \, d\bdx .
\end{equation*}
Using \eqref{canonicalKernelIntegral}, we immediately have the upper
bound:
\begin{eqnarray*}
2 \int_{\OmegaNLC} \int_{\OmegaNLC} \myChi \, u^2(\bdx)~d\bdxp \, d\bdx
& \leq & 2 w_d \, \delta^d ~\|u\|_{L_2(\OmegaNLC)}^2 \\
& =: & \overline{\lambda} \, \delta^d ~\|u\|_{L_2(\OmegaNLC)}^2.
\end{eqnarray*}
\end{proof}

\begin{rem}
The upper bound is \emph{sharp}; see the companion article
\cite[Sect. 4]{Aksoylu:2010:NLBVP}.  We also numerically demonstrate the
numerical sharpness of the upper bound; see \S\ref{sec:auu_num_expts}.
\end{rem}

\begin{rem}
A proof similar to that of Lemma \eqref{lem:upperBound} can be given
to show the boundedness of the bilinear form with an explicit constant. Namely,
\begin{equation} \label{bilinearFormBdd}
|a(u,v)| \leq 2 \, w_d\, \delta^d \,
\|u\|_{L_2(\OmegaNLC)} \|v\|_{L_2(\OmegaNLC)}.
\end{equation}
Also see the companion article \cite{Aksoylu:2010:NLBVP} for the boundedness
of the bilinear form for general kernel functions.
\end{rem}

\subsection{The Conditioning of the Stiffness Matrix 
$K$} \label{sec:cond_auu}

Combining the refined nonlocal Poincar\'{e} inequality \eqref{poincareDelta}
and the upper bound \eqref{eq:upperBound}, we arrive at a condition number estimate.
\begin{theorem} \label{thm:spec_equiv}
For sufficiently small $\delta$, the following spectral equivalence holds:
\begin{equation} \label{spectralEquivK}
\lambda_{refined}(\OmegaNLC) ~ \delta^{d+2} \leq
\frac{a(u,u)}{\norm{u}_{L_2(\OmegaNLC)}^2}
\leq \overline{\lambda}~\delta^d, \quad u \in L_{2, 0}(\OmegaNLC).
\end{equation}
Let $K$ be the stiffness matrix produced by discretizing $a(u,u)$. Then, the condition number of $K$ has the bound:
\begin{equation} \label{condNumberA}
\kappa(K) \lesssim \delta^{-2}.
\end{equation}
\end{theorem}
The spectral equivalence \eqref{spectralEquivK} enables us to construct an $h$ \emph{independent} upper bound for the condition number.

Note that the condition number of the stiffness matrix also depends upon the mesh size $h$. As an illustration, consider that the nonlocal bilinear form $a(u,u)$ must converge to the corresponding local bilinear form in the limit $\delta \rightarrow 0$, as demonstrated in \S\ref{sec:choose_m}, and that the condition number of the associated local stiffness matrix varies with $h^{-2}$. Thus, the bound in \eqref{condNumberA} is not tight in this limit, but allows us to investigate the highly nonlocal regime $h \ll \delta$, which is our principle interest. For an alternative approach to quantifying mesh-dependence in the conditioning of nonlocal models, see \cite{Zhou:2010:MathAnalPD}.

\begin{table}[ht!]
\centering
\subtable[Fixed $\delta$, vary $h$.]{
\scriptsize{
\begin{tabular}{|r|c|c|c|c|c|c|c|}
  \hline
                              &                                  & \multicolumn{3}{|c|}{Piecewise Constant Shape Functions} & \multicolumn{3}{|c|}{Piecewise Linear Shape Functions} \\
  \multicolumn{1}{|c|}{$1/h$} & \multicolumn{1}{|c|}{$1/\delta$} & \multicolumn{1}{|c|}{$\lambda_{\min}$} & \multicolumn{1}{|c|}{$\lambda_{\max}$} & \multicolumn{1}{|c|}{Condition \#} & \multicolumn{1}{|c|}{$\lambda_{\min}$} & \multicolumn{1}{|c|}{$\lambda_{\max}$} & \multicolumn{1}{|c|}{Condition \#} \\
  \hline
     2000 & 20 & 1.94E-07 & 6.07E-05 & 3.13E+02 & 1.94E-07 & 6.07E-05 & 3.13E+02 \\
     4000 & 20 & 9.69E-08 & 3.04E-05 & 3.13E+02 & 9.69E-08 & 3.04E-05 & 3.14E+02 \\
     8000 & 20 & 4.84E-08 & 1.52E-05 & 3.14E+02 & 4.84E-08 & 1.52E-05 & 3.14E+02 \\
  \hline
\end{tabular}
}
\label{table:auu:h}
}
\subtable[Fixed $h$, vary $\delta$.]{
\scriptsize{
\begin{tabular}{|r|c|c|c|c|c|c|c|}
  \hline
                              &                                  & \multicolumn{3}{|c|}{Piecewise Constant Shape Functions} & \multicolumn{3}{|c|}{Piecewise Linear Shape Functions} \\
  \multicolumn{1}{|c|}{$1/h$} & \multicolumn{1}{|c|}{$1/\delta$} & \multicolumn{1}{|c|}{$\lambda_{\min}$} & \multicolumn{1}{|c|}{$\lambda_{\max}$} & \multicolumn{1}{|c|}{Condition \#} &
  \multicolumn{1}{|c|}{$\lambda_{\min}$} & \multicolumn{1}{|c|}{$\lambda_{\max}$} & \multicolumn{1}{|c|}{Condition \#} \\
  \hline
    8000 & 20 & 4.84E-08 & 1.52E-05 & 3.15E+02 & 4.84E-08 & 1.52E-05 & 3.14E+02 \\
    8000 & 40 & 6.24E-09 & 7.61E-06 & 1.22E+03 & 6.24E-09 & 7.60E-06 & 1.22E+03 \\
    8000 & 80 & 7.92E-10 & 3.80E-06 & 4.80E+03 & 7.91E-10 & 3.80E-06 & 4.80E+03 \\
  \hline
\end{tabular}
}
\label{table:auu:delta}
}
\caption{Condition number for $K$ in 1D for (a) fixed $\delta$, allowing $h$ to vary, and (b) fixed $h$, allowing $\delta$ to vary, for both piecewise constant and linear shape functions. We see that the conditioning is apparently not strongly influenced by the choice of shape function. This data is plotted in Figures \ref{fig:auu}.} \label{table:auu}
\end{table}

\begin{figure}[ht!]
\centering
\subfigure[Fixed $\delta$, vary $h$.]{
 \scalebox{0.42}{\includegraphics{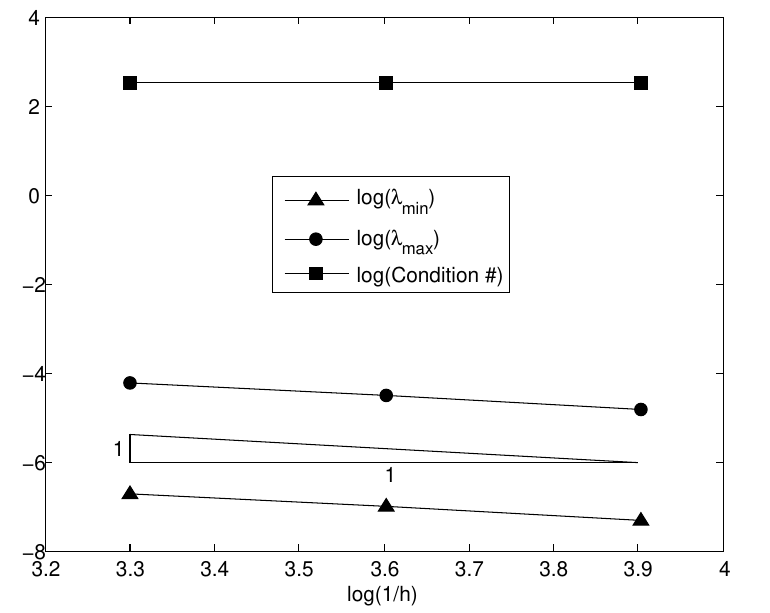}}
 \label{fig:auu:h}
}
\subfigure[Fixed $h$, vary $\delta$.]{
 \scalebox{0.42}{\includegraphics{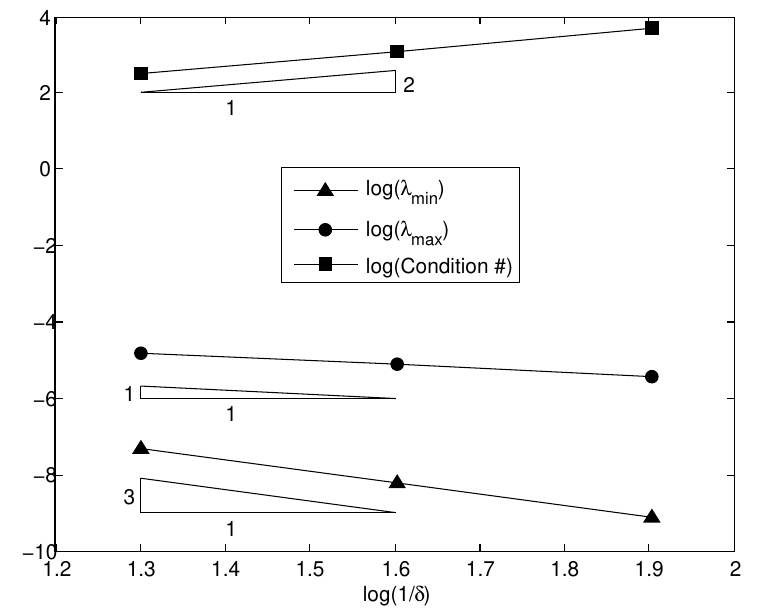}}
 \label{fig:auu:delta}
}
\caption{Condition number for $K$ in 1D for (a) fixed $\delta$, allowing $h$ to vary, and (b) fixed $h$, allowing $\delta$ to vary. The condition number is only weakly $h$-dependent, but varies with $\delta^{-2}$. These figures are plotted from data in Table \ref{table:auu}. The plots for piecewise linear and piecewise constant shape functions are identical.}
\label{fig:auu}
\end{figure}

\subsection{Numerical Verification of Condition Number by a Finite Element Formulation}
\label{sec:auu_num_expts}

For all computational results in this article, we let $\Omega = [0,1]^d$ be the unit $d$-cube, where $d$ is the spatial dimension, with $\OmegaNLC= [-\delta,1+\delta]^d$ the nonlocal closure. We impose the Dirichlet boundary condition $u=0$ on $\mathcal{B}\Omega = \OmegaNLC \backslash \Omega$. We use a conforming triangulation $\mathcal{T}_h$ where each element $E$ of $\mathcal{T}_h$ is a $d$-cube with a side length $h>0$. Consequently, each element in 1D, 2D, and 3D is a line segment of length $h$, a square of area $h^2$, and a cube of volume $h^3$, respectively.  Let $V_h \subset V$ be a finite dimensional subspace of $V$ from \eqref{eq:pdFuncSpace}. We use
a Galerkin finite element formulation of \eqref{eq:pdEquMotionWeak}:
\begin{align} \label{eq:pdEquMotionWeakDiscrete}
 a(u_h,v_h) = (b,v_h) \qquad \forall v_h \in V_h,
\end{align}
with Dirichlet boundary condition $u_h = 0$ on $\mathcal{B}\Omega$, where $V_h$ is the space of piecewise constant or piecewise linear shape functions on the mesh $\mathcal{T}_h$, and where we employ the canonical kernel function $\chi_\delta$ from \eqref{eq:pdKernel}. We denote by $K$ the stiffness matrix arising from the left-hand side of \eqref{eq:pdEquMotionWeakDiscrete}. To verify our theoretical results we numerically determine the ratio of the largest and smallest eigenvalues of $K$, defining the condition number of the problem.

\begin{table}[t]
\centering
\subtable[Fixed $\delta$, vary $h$.]{
\scriptsize{
\begin{tabular}{|r|c|c|c|c|}
  \hline
  \multicolumn{1}{|c|}{$1/h$} & \multicolumn{1}{|c|}{$1 / \delta$} & \multicolumn{1}{|c|}{$\lambda_{\min}$} & \multicolumn{1}{|c|}{$\lambda_{\max}$} & \multicolumn{1}{|c|}{Condition \#} \\
  \hline
     50  & 10 & 2.95E-07 & 1.40E-05 & 4.77E+01 \\
    100  & 10 & 7.11E-08 & 3.54E-06 & 4.97E+01 \\
    200  & 10 & 1.75E-08 & 8.86E-07 & 5.05E+01 \\
  \hline
\end{tabular}
}
\label{table2d:auu:h}
}
\subtable[Fixed $h$, vary $\delta$.]{
\scriptsize{
\begin{tabular}{|r|c|c|c|c|}
  \hline
  \multicolumn{1}{|c|}{$1/h$} & \multicolumn{1}{|c|}{$1 / \delta$} & \multicolumn{1}{|c|}{$\lambda_{\min}$} & \multicolumn{1}{|c|}{$\lambda_{\max}$} & \multicolumn{1}{|c|}{Condition \#} \\
  \hline
    200  & 10 & 1.75E-08 & 8.86E-07 & 5.05E+01 \\
    200 &  20 & 1.17E-09 & 2.22E-07 & 1.90E+02 \\
    200 &  40 & 7.63E-11 & 5.50E-08 & 7.21E+02 \\
  \hline
\end{tabular}
}
\label{table:auu2d:delta}
}
\caption{Condition number for $K$ in 2D using piecewise constant shape functions for (a) fixed $\delta$, allowing $h$ to vary, and (b) fixed $h$, allowing $\delta$ to vary. This data is plotted in Figure \ref{fig:auu2d}.} \label{table:auu2d}
\end{table}

\begin{figure}[t]
\centering
\subfigure[Fixed $\delta$, vary $h$. ]{
 \scalebox{0.42}{\includegraphics{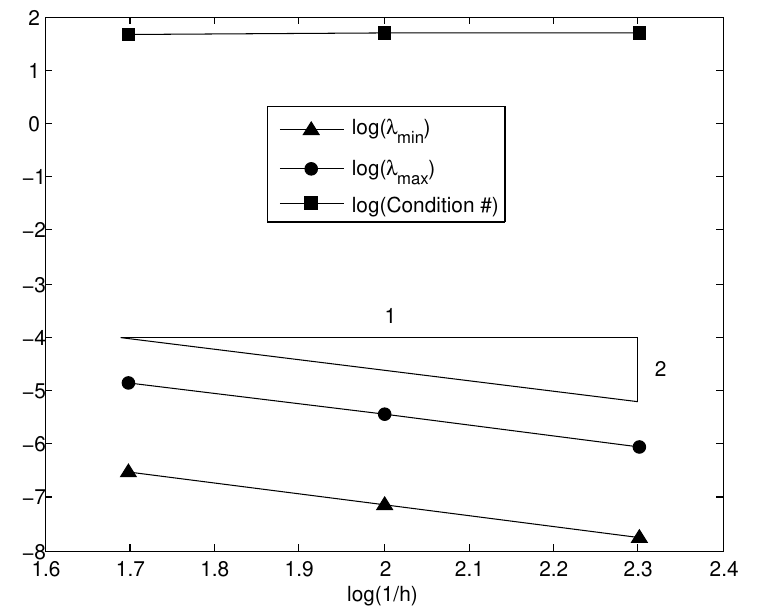}}
 \label{fig:auu2d:h}
}
\subfigure[Fixed $h$, vary $\delta$.]{
 \scalebox{0.42}{\includegraphics{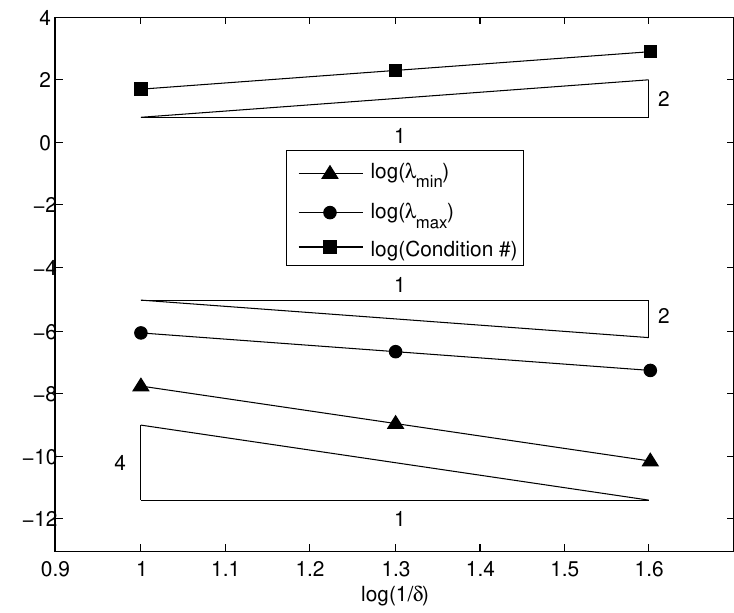}}
 \label{fig:auu2d:delta}
}
\caption{Condition number for $K$ in 2D for (a) fixed $\delta$, allowing $h$ to vary, and (b) fixed $h$, allowing $\delta$ to vary. The condition number is only weakly h-dependent, but varies with $\delta^{-2}$. These figures are plotted from data in Table \ref{table:auu2d}.}
\label{fig:auu2d}
\end{figure}

\subsubsection{Results in One Dimension} \label{sec:auu1D}

Results in this section appear in Tables \ref{table:auu} and Figures \ref{fig:auu}, where we consider the $h \ll \delta$ regime. We show results for both piecewise constant and piecewise linear shape functions to verify that the choice of shape function apparently does not influence the conditioning of the discrete system. We first compute the condition number of $K$ for different $h$ while holding $\delta$ fixed, and observe that the condition number of $K$ is only weakly $h$-dependent. The minimum and the maximum eigenvalues depend linearly on $h$, with a slope of nearly unity. We then compute the condition number of $K$ for different values of $\delta$ while holding $h$ fixed, and observe that the condition number varies with $\delta^{-2}$. Further, the maximum eigenvalue is proportional to $\delta$, in agreement with Lemma \ref{lem:upperBound}. Lastly, the minimum eigenvalue varies as $\delta^{3}$, in agreement with \eqref{poincareDelta} and our finding of $m=3$ in \S\ref{sec:choose_m}. This suggests that, in one dimension, we should redefine $C(x,x^\prime)$ in \eqref{eq:pdKernel} as
\begin{align*}
C(x,x^\prime) = \left\{
\begin{array}{cl}
\delta^{-3}, & \norm{x-x^\prime} \leq \delta \\
0, & \textrm{otherwise.}
\end{array}
\right.,
\end{align*}
for consistency with the weak form of the classical (local) Laplace operator in the limit $\delta \rightarrow 0$.

\subsubsection{Results in Two Dimensions}

Results in this section appear in Tables \ref{table:auu2d} and Figures \ref{fig:auu2d}. We consider only piecewise constant shape functions in 2D. We first compute the condition number of $K$ for different $h$ while holding $\delta$ fixed, and observe that minimum and maximum eigenvalues depend linearly on $h$ with a slope of approximately two, and again the condition number of $K$ depends only weakly upon the mesh size. We then compute the condition number of $K$ for different values of $\delta$ while holding $h$ fixed, and observe that the condition number again varies as $\delta^{-2}$, in agreement with \eqref{condNumberA}. Further, the maximum eigenvalue is proportional to $\delta^2$ in agreement with Lemma \ref{lem:upperBound}, and the minimum eigenvalue is proportional to $\delta^{4}$ in agreement with \eqref{poincareDelta}.

\section{A Nonlocal Two-Domain Problem} \label{sec:DDPD}

\begin{figure}
 \begin{center}
 \psfrag{1}{$\Omega_1$}  \psfrag{2}{$\Omega_2$}  \psfrag{x}{$\Gamma$}  \psfrag{4}{$\mathcal{B} {\Omega_2}$}  \psfrag{5}{$\Gamma_1$}  \psfrag{6}{$\Gamma_2$}  \psfrag{7}{$\mathcal{B} {\Omega_1}$}  \psfrag{8}{$\Gamma_N$} \psfrag{9}{$\Gamma_S$}
 \scalebox{0.75}{\includegraphics{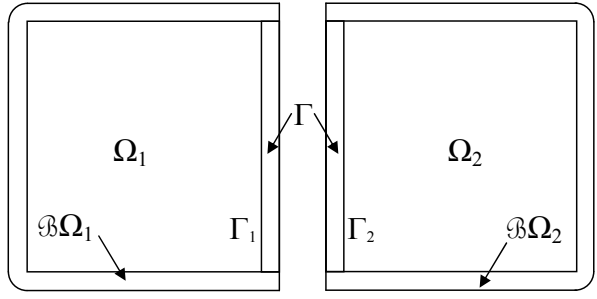}}
 \end{center}
 \caption{A nonlocal two-domain problem. This is a decomposition of the domain $\Omega$ in Figure \ref{fig:1Domain} into overlapping subdomains $\Omega^{(1)}$, $\Omega^{(2)}$, and $\mathcal{B}\Omega$ into overlapping nonlocal boundaries $\mathcal{B}\Omega_1$, $\mathcal{B}\Omega_2$. Note that the interface $\Gamma$ is $d=2$-dimensional.}\label{fig:2Domain}
\end{figure}

We will construct a weak (variational) formulation for nonlocal
domain decomposition.  We first identify the pieces of the domain for
this decomposition.  Consider the domain in Figure
\ref{fig:2Domain}. The nonlocal boundary of $\Omega$, $\mathcal{B}
\Omega$, is defined to be the closed region of thickness $\delta$
surrounding $\Omega$.  Let $\Gamma$ be the open region corresponding to
the interface between the two overlapping open subdomains
$\Omega^{(1)}$ and $\Omega^{(2)}$. We define the overlapping
subdomains $\Omega^{(i)},~~i=1, 2$, as the following:
\begin{equation*}
 \Omega^{(i)} := \Omega_i \cup \Gamma \cup \Gamma_i,
\end{equation*}
where $\Gamma_i$ is the open line segment adjacent to
$\Omega_i$ and $\Gamma$. Let $\mathcal{B}\Omega_i$ be the nonlocal closed
boundary of $\Omega_i$ that intersects $\mathcal{B}\Omega$.
The main domain decomposition contributions of this article, namely,
the equivalence of the one-domain weak and two-domain weak forms will
be proved next.

\subsection{Two-Domain Variational Form}
We present a two-domain weak formulation
of \eqref{eq:pdEquMotionWeak} and prove its equivalence to the
original single-domain formulation \eqref{eq:pdEquMotionWeak}.
We define the spaces, $i= 1, 2$,
\begin{align} \label{eqn:V_super_i}
 V^{(i)}   &:= \left\{ v \in L_2(\OmegaNLCi) \; : \; v \vert_{\mathcal{B} \Omega_{i}} = 0 \right\},\\
\nonumber V^{(i),0} &:= \left\{ v \in L_2(\OmegaNLCi) \; : \; v \vert_{\mathcal{B} \Omega_{i} \cup \Gamma \cup \Gamma_i} = 0 \right\},\\
\nonumber \Lambda & := \left\{ \mu \in L_2(\Gamma): \mu = v \vert_\Gamma~ \text{for some suitable}~ v \in L_{2,0}(\OmegaNLC) \right\}.
\end{align}
We can reduce the outer domain of integration in the bilinear form from $\OmegaNLC$ to $\Omega$ by taking advantage of the zero Dirichlet boundary condition. Namely,
\begin{eqnarray}
a(u,v) & = &  -
\int_{\OmegaNLC} \left\{ \int_{\OmegaNLC} \myChi ~
[u(\bdxp) - u(\bdx)]~d\bdxp \right\} ~ v(\bdx)~d\bdx \nonumber \\
& = &  - \int_{\Omega} \left\{ \int_{\OmegaNLC} \myChi ~
[u(\bdxp) - u(\bdx)]~d\bdxp \right\} ~ v(\bdx)~d\bdx, \quad v \in V
\label{reducedBilinearForm}.
\end{eqnarray}
Therefore, our construction is based on the reduced bilinear form
\eqref{reducedBilinearForm}. We further define a bilinear form
$a_{\Omega^{(i)}}(u,v): V \times V \rightarrow \REAL$ as follows:
\begin{eqnarray} \nonumber a_{\Omega^{(i)}}(u,v) & := & -
\int_{\Omega_i} \left\{ \int_{\Omega^{(i)} \cup \mathcal{B} \Omega_{i}}
\chi_\delta(\bdx - \bdxp) ~
[u(\bdxp) - u(\bdx)]~d\bdxp \right\} ~ v(\bdx)~d\bdx \\ \label{eq:bilinearForm}
& & - \frac{1}{2} \int_{\Gamma} \left\{ \int_{\OmegaNLC}
\chi_\delta(\bdx - \bdxp) ~
[u(\bdxp) - u(\bdx)]~d\bdxp \right\} ~ v(\bdx)~d\bdx.
\end{eqnarray}

We utilize the following notation to suppress the integrals in \eqref{eq:bilinearForm}:
\begin{eqnarray}
 a_{\Omega_i}(u,v) &:=& - \int_{\Omega_i} \left\{ \int_{\Omega^{(i)} \cup \mathcal{B} \Omega_{i}} \chi_\delta(\bdx - \bdxp) ~ [u(\bdxp) - u(\bdx)]~d\bdxp \right\} ~ v(\bdx)~d\bdx \\
 a_{\Gamma}(u,v) & := & - \int_{\Gamma} \left\{ \int_{\OmegaNLC} \chi_\delta(\bdx - \bdxp) ~ [u(\bdxp) - u(\bdx)]~d\bdxp \right\} ~ v(\bdx)~d\bdx
\end{eqnarray}
We can now represent the bilinear form \eqref{eq:bilinearForm} as:
$$
a_{\Omega^{(i)}} (u,v) = \frac{1}{2}~a_\Gamma(u,v) + a_{\Omega_i}(u,v).
$$

\begin{rem}
The test function $v_i = v \vert_{\Omega_i} \in V^{(i),0},~i=1, 2$ has its support only in $\Omega_i$ not $\Omega^{(i)}$. Hence, we may reduce the bilinear form \eqref{eq:bilinearForm} to
\begin{equation}
 \label{eq:bilinearFormReduced}
 a_{\Omega^{(i)}}(u^{(i)},v_i) =  a_{\Omega_i}(u^{(i)},v_i).
\end{equation}
Although, $a_{\Gamma}(u^{(i)},v_i)$ may appear to create a coupling between the subdomains, no such coupling exists because $v_i$ vanishes on $\Gamma$.  Therefore, subdomain condition \eqref{eq:DCi} is an expression only for subdomain $\Omega^{(i)}$.
\end{rem}

Now, we state the two-domain weak form following the notation of \cite{Quarteroni:1999:DDforPDEs}: Find $u^{(i)} \in V^{(i)}$, $i=1, 2$:
\begin{subequations} \label{eq:2DomainPD}
\begin{align}
a_{\Omega^{(i)}}(u^{(i)},v_i) & = (b,v_i)_{\Omega_i} & \forall v_i \in V^{(i),0}, \label{eq:DCi}\\
u^{(1)} & =  u^{(2)} & \mbox{ on } \overline{\Gamma}, \label{eq:TCtrivial}\\
\sum_{i=1, 2} a_{\Omega^{(i)}}(u^{(i)},\mathcal{R}^{(i)} \mu) & = (b,\mu)_{\Gamma} + \sum_{i=1, 2} (b,\mathcal{R}^{(i)}\mu)_{\Omega_i} & \forall \mu \in \Lambda. \label{eq:TCcombi}
\end{align}
\end{subequations}
where $\mathcal{R}^{(i)}$ denotes any possible extension operator from
$L_2(\Gamma)$ to $V^{(i)}$. An extension operator
$\mathcal{R}^{(i)}: L_2(\Gamma) \rightarrow V^{(i)}$ is defined to be an operator
which satisfies $(\mathcal{R}^{(i)} \eta )\vert_\Gamma = \eta$ for
$\eta \in L_2(\Gamma).$
Next, we will show that the one- and two-domain weak forms are equivalent. The proof for the local case can be found in \cite[Lemma 1.2.1]{Quarteroni:1999:DDforPDEs}.

\begin{lem} \label{1DomainWeak_2DomainWeak}
The problems \eqref{eq:pdEquMotionWeak} and \eqref{eq:2DomainPD} are equivalent.\
\end{lem}
\begin{proof}
$\eqref{eq:pdEquMotionWeak} \Rightarrow \eqref{eq:2DomainPD}:$\\
Let $u^{(i)} = u \vert_{\Omega^{(i)}}\in V^{(i)}$ and $v_i = v \vert_{\Omega_i} \in V^{(i),0},~i=1, 2$. Extend these functions by zero extension;
\begin{eqnarray*}
\theta^{(i)} u^{(i)} & := &  \left\{
\begin{array}{ll}
u^{(i)},  & \text{in}~\Omega^{(i)} \\
0, & \text{otherwise}
 \end{array}\right.\\
\theta_i v_i & := & \left\{
\begin{array}{ll}
v_i,  & \text{in}~\Omega_i \\
0, & \text{otherwise}.
 \end{array}\right.
\end{eqnarray*}
By LHS of \eqref{eq:pdEquMotionWeak} and using $v_i \vert_{\Gamma} = 0$:
\begin{eqnarray*}
 a(\theta^{(i)} u^{(i)}, \theta_i v_i) & = & - \int_\Omega \left\{ \int_{\OmegaNLC} \chi_\delta(\bdx - \bdxp) ~ [\theta^{(i)} u^{(i)}(\bdxp) - \theta^{(i)} u^{(i)}(\bdx)]~d\bdxp \right\} ~ \theta_i v_i(\bdx)~d\bdx \\
 &=& a_{\Omega_i} (u^{(i)}, v_i) \\
 &=& \frac{1}{2}~a_{\Gamma}(u^{(i)},v_i) + a_{\Omega_i} (u^{(i)}, v_i) \\
 &=& a_{\Omega^{(i)}}(u^{(i)}, v_i)
\end{eqnarray*}
By RHS of \eqref{eq:pdEquMotionWeak},
\begin{equation*}
(b, \theta_i v_i) =  (b,v_i)_{\Omega_i}.
\end{equation*}
Hence, \eqref{eq:DCi} is satisfied. \eqref{eq:TCtrivial} is trivially satisfied.

Further, for $\mu \in \Lambda$ define the function $\mathcal{R}\mu$ as:
\begin{align*}
 \mathcal{R}\mu := \left\{
 \begin{array}{ll}
 \mathcal{R}^{(1)}\mu, & \mbox{in } \Omega^{(1)} \\
 \mathcal{R}^{(2)}\mu, & \mbox{in } \Omega^{(2)}.
 \end{array}\right.\
\end{align*}
Since $\mathcal{R}\mu$ lives only in $\Omega_1 \cup \Gamma_1 \cup \Gamma \cup \Gamma_2 \cup \Omega_2$, it vanishes on $\mathcal{B} \Omega$. Therefore, $\mathcal{R}\mu \in V$.

From \eqref{eq:pdEquMotionWeak}, partitioning the outer integral and using $\mathcal{R}^{(1)}\mu =  \mathcal{R}^{(2)}\mu = \mu$ on $\Gamma$, we obtain the LHS of \eqref{eq:TCcombi}:
\begin{eqnarray*}
 a(u,\mathcal{R} \mu ) &=& \frac{1}{2}~a_{\Gamma}(u^{(1)}, \mu) + \frac{1}{2}~a_{\Gamma}(u^{(2)}, \mu) + \sum_{i=1, 2} a_{\Omega_i}(u^{(i)}, \mathcal{R}^{(i)} \mu)\\
 &=& \frac{1}{2}~a_{\Gamma} (u^{(1)}, \mathcal{R}^{(1)} \mu) + \frac{1}{2}~a_{\Gamma} (u^{(2)}, \mathcal{R}^{(2)} \mu) + \sum_{i=1, 2} a_{\Omega_i} (u^{(i)}, \mathcal{R}^{(i)} \mu) \\
 &=& a_{\Omega^{(1)}}(u^{(1)},\mathcal{R}^{(1)} \mu) +
a_{\Omega^{(2)}}(u^{(2)},\mathcal{R}^{(2)} \mu).
\end{eqnarray*}
Likewise, from \eqref{eq:pdEquMotionWeak} and partitioning the integral, we obtain the RHS of \eqref{eq:TCcombi}:
$$
(b,\mathcal{R} \mu)_\Omega =
(b,\mathcal{R}^{(1)}\mu)_{\Omega_1} + (b,\mathcal{R}^{(2)} \mu)_{\Omega_2} +
(b,\mu)_{\Gamma}.
$$
Hence, we obtain the transmission condition \eqref{eq:TCcombi}.
\newline
\linebreak
$\eqref{eq:2DomainPD} \Rightarrow \eqref{eq:pdEquMotionWeak}:$\\
Let $u_\Gamma := u^{(1)}\vert_\Gamma$ (due to \eqref{eq:TCtrivial}, we also have $u_\Gamma = u^{(2)}\vert_\Gamma$) and
\begin{align} \label{defn:u}
u := \left\{
\begin{array}{ll}
u^{(1)}, & \mbox{in } \Omega_1 \\
u^{(2)}, & \mbox{in } \Omega_2 \\
u_\Gamma, & \mbox{in } \Gamma. \\
\end{array}\right.\
\end{align}
We partition the outer integral, use \eqref{defn:u} and the transmission condition \eqref{eq:TCtrivial}. Then, for $v \in V$, LHS in
\eqref{eq:pdEquMotionWeak} becomes the following:
\begin{eqnarray}
a(u, v) & = &
\frac{1}{2}~a_{\Gamma} (u, v) + \frac{1}{2}~a_{\Gamma} (u, v) +
\sum_{i=1, 2} a_{\Omega_i} (u, v) \nonumber \\
& = &
\frac{1}{2}~a_{\Gamma} (u_\Gamma, v) + \frac{1}{2}~a_{\Gamma} (u_\Gamma, v) +
\sum_{i=1, 2} a_{\Omega_i} (u^{(i)}, v) \nonumber \\
& = & \frac{1}{2}~a_{\Gamma} (u^{(1)}, v) + \frac{1}{2}~a_{\Gamma} (u^{(2)}, v) +
\sum_{i=1, 2} a_{\Omega_i} (u^{(i)}, v) \nonumber \\
& = & \sum_{i=1, 2} a_{\Omega^{(i)}}(u^{(i)},v).  \label{eq:LHS1Domain}
\end{eqnarray}
Let $\mu := v \vert_\Gamma$. Then, $v - \mathcal{R}^{(i)}\mu \in V^{(i),0}$.  First, we add and subtract $\mathcal{R}^{(i)} \mu$ to the second slot of the bilinear form in \eqref{eq:LHS1Domain} and apply the domain conditions \eqref{eq:DCi} for $v - \mathcal{R}^{(i)}\mu$. Then, we apply the transmission condition \eqref{eq:TCcombi} and use $v \vert_\Gamma = \mu$. Hence, we arrive at the RHS in \eqref{eq:pdEquMotionWeak}:
\begin{eqnarray*}
\sum_{i=1, 2} a_{\Omega^{(i)}}(u^{(i)},v) & = & \sum_{i=1, 2}
a_{\Omega^{(i)}}(u^{(i)}, v - \mathcal{R}^{(i)} \mu)
+ \sum_{i=1, 2} a_{\Omega^{(i)}}(u^{(i)}, \mathcal{R}^{(i)} \mu) \\
& = & \sum_{i=1, 2} (b, v - \mathcal{R}^{(i)} \mu)_{\Omega_i} +
\sum_{i=1, 2} a_{\Omega^{(i)}}(u^{(i)}, \mathcal{R}^{(i)} \mu) \\
& = & \sum_{i=1, 2} (b, v - \mathcal{R}^{(i)} \mu)_{\Omega_i} +
(b,\mu)_{\Gamma} + \sum_{i=1, 2} (b,\mathcal{R}^{(i)}\mu)_{\Omega_i}
\\
& = & (b,\mu)_{\Gamma} + \sum_{i=1, 2} (b,v)_{\Omega_i} \\
& = & (b,v).
\end{eqnarray*}
\end{proof}


\section{Towards Nonlocal Substructuring} \label{sec:NonlocalSubstructure}

Here we write out the linear algebraic representations arising from
the two-domain weak form \eqref{eq:2DomainPD}, identifying the
discrete subdomain equations and transmission conditions. We then
construct a nonlocal Schur complement, discuss its condition number as
a function of $h$, $\delta$, and provide supporting numerical
experiments.

\subsection{Linear Algebraic Representations}
We consider a finite element discretization of \eqref{eq:2DomainPD}.
Letting $V_h^{(i)}$ denote the finite element space corresponding to
$\Omega^{(i)}$, we define:
\begin{eqnarray*}
V_h^{(i),0} & := & \left\{ v_h \in V_h^{(i)} :~ v_{h}
\vert_{\mathcal{B} \Omega_i \cup \Gamma \cup \Gamma_i}= 0 \right\}\\
\Lambda_h & := & \left\{  \mu_{h} \in L_2(\Gamma):~\mu_{h} = v_h\vert_{\Gamma}~\text{for some suitable}~ v_h \in V_h \right\}.
\end{eqnarray*}
Here, $\Lambda_h$ denotes a finite element discretization of $L_2(\Gamma)$.
We see that the finite element formulation of \eqref{eq:2DomainPD}
can be written as:
\begin{subequations}
\begin{align} \label{eqn:2domainFEM_1}
a_{\Omega^{(i)}}(u_h^{(i)},v_{i,h}) & =  (b,v_{i,h})_{\Omega_i} & \forall v_{i,h} \in V_h^{(i),0}, \\ \label{eqn:2domainFEM_2}
u_h^{(1)} & =  u_h^{(2)} & \mbox{ on } \overline{\Gamma},\\ \label{eqn:2domainFEM_3}
\sum_{i=1, 2} a_{\Omega^{(i)}}(u_h^{(i)},\mathcal{R}_{h}^{(i)} \mu_h) & = (b,\mu_h)_{\Gamma} + \sum_{i=1, 2} (b,\mathcal{R}_h^{(i)}\mu_h)_{\Omega_i} & \forall \mu_h \in \Lambda_h.
\end{align}
\end{subequations}
where $\mathcal{R}_h^{(i)}$ denotes any possible extension operator from $\Gamma_h$ to $V_h^{(i)}$. Following standard practice, we take these extension operators to be the finite element interpolant, which is defined to be equal to $\mu_h$ at the nodes in the thick interface $\Gamma$ and zero on the internal nodes of $\Omega_i$. If we number nodes in $\Omega_1$ first, nodes in $\Omega_2$ second, and nodes in $\Gamma$ last, we will arrive at a global stiffness matrix that takes the traditional block arrowhead form:
\begin{align} \label{stiffnessMatrix}
  K = \left[
        \begin{array}{ccc}
          K_{11} & 0 & K_{1 \Gamma} \\
          0 & K_{22} & K_{2 \Gamma} \\
          K_{\Gamma 1} & K_{\Gamma 2} & K_{\Gamma \Gamma} \\
        \end{array}
      \right]
      \left[
        \begin{array}{c}
          u_1 \\
          u_2 \\
          u_\Gamma\\
        \end{array}
      \right] =
      \left[
        \begin{array}{c}
          f_1 \\
          f_2 \\
          f_\Gamma\\
        \end{array}
      \right].
\end{align}
The first two block rows of the matrix in \eqref{stiffnessMatrix} arise from discretizing \eqref{eqn:2domainFEM_1}, and the last block row arises from discretizing \eqref{eqn:2domainFEM_3}.

\subsection{Discrete Energy Minimizing Extension and the Schur
Complement Conditioning}

In order to study the conditioning of the Schur complement in the
nonlocal setting, we define an analog of the discrete harmonic
extension in the local case. 
\begin{defn} \label{defn:EME}
For a given $q \in \Lambda_h$,
$E_i: \Lambda_h \rightarrow V_h^{(i)}$ defines a discrete energy
minimizing extension into $\Omega_i$,~if
\begin{eqnarray} \label{defn:EMEprop1}
E_i(q)|_\Gamma & = & q,\\
a_i(E_i(q),v) & = & 0, \quad v \in V_h^{(i),0}, \nonumber
\end{eqnarray}
where $a_i(\cdot, \cdot)$ denotes the bilinear form restricted
to $\OmegaNLCi$. Namely,
\begin{equation*}
a_i(u,v) =
\int_{\OmegaNLCi} \left\{ \int_{\OmegaNLCi} \myChi ~
[u(\bdxp) - u(\bdx)]~d\bdxp \right\} ~ v(\bdx)~d\bdx.
\end{equation*}
\end{defn}

The energy minimizing extension $E_i(q)$ of $q$ defines a canonical
bilinear form $s_i(q,q) : \Lambda_h \times \Lambda_h \rightarrow
\REAL$ that is associated to the interface $\Gamma$ whose
discretization corresponds to the subdomain Schur complement matrix
$S^{(i)}$ below. Let $\underline{q}$ denote the vector representation
of $q$.
\begin{eqnarray} \label{defn:s}
s_i(q,q) & := & a_i(E_i(q),E_i(q))\\
\underline{q}^t S^{(i)} \underline{q} & = & a_i(E_i(q_h),E_i(q_h)).
\label{SchurComplement}
\end{eqnarray}

Let us denote the restriction of $u \in V_h^{(i)}$ to
$\Gamma$ by $u_\Gamma := u|_\Gamma$. The following discussion will
reveal the reason why $E_i(u_\Gamma)$ is called an energy minimizing
extension. Let us consider the following decomposition of $u$:
\begin{equation} \label{eqn:decompositionEME}
u = [u - E_i(u_\Gamma)] + E_i(u_\Gamma).
\end{equation}
Since $\left( u - E_i(u_\Gamma) \right)|_\Gamma = 0$, by
Definition \ref{defn:EME} we have:
\begin{equation} \label{eqn:aOrthogonality}
a_i(u - E_i(u_\Gamma), E_i(u_\Gamma)) = 0.
\end{equation}
Using \eqref {eqn:decompositionEME} and \eqref{eqn:aOrthogonality}, we have
the energy minimizing property of $E_i(u_\Gamma)$ among
$u \in  V_h^{(i)}$ with $u|_\Gamma = u_\Gamma$:
\begin{eqnarray}
a_i(u,u) & = & a_i(u - E_i(u_\Gamma), u - E_i(u_\Gamma)) +
2 \, a_i(u - E_i(u_\Gamma), E_i(u_\Gamma)) + a_i(E_i(u_\Gamma), E_i(u_\Gamma))
\nonumber \\
& \geq & a_i(E_i(u_\Gamma), E_i(u_\Gamma)) \label{ineq:minPropEME}.
\end{eqnarray}
Therefore, using \eqref{ineq:minPropEME}, \eqref{defn:s} and
\eqref{eq:upperBound}, we have:
\begin{eqnarray*}
s_i(u_\Gamma, u_\Gamma)
\leq a_i(u,u) \leq \overline{\lambda} ~ \delta^d ~\norm{u}_{L_{2}(\OmegaNLCi)}^2,
\end{eqnarray*}
for all $u \in V_h^{(i)}$, in particular, for $u = u_\Gamma$.
Hence,
\begin{equation}
s_i(u_\Gamma, u_\Gamma) \leq
\overline{\lambda} ~ \delta^d ~\norm{u}_{L_{2}(\Gamma)}^2.
\end{equation}
For the lower bound, we simply use \eqref{defn:EMEprop1} and
\eqref{poincareDelta}:
\begin{equation} \label{sLowerBd}
\lambda_{refined} \, \delta^{d+2} \norm{u}_{L_2(\Gamma)}^2 \leq
\lambda_{refined} \, \delta^{d+2} \norm{E_i(u_\Gamma)}_{L_2(\OmegaNLCi)}^2 \leq
a_i(E_i(u_\Gamma), E_i(u_\Gamma)) = s_i(u_\Gamma, u_\Gamma).
\end{equation}
We have proved the following spectral equivalence result:
\begin{theorem}
For any $q \in \Lambda_h \subset L_2(\Gamma)$, we have:
\begin{equation} \label{spectranEquivS}
\lambda_{refined} ~ \delta^{d+2} \leq
\frac{s_i(q, q)}{\norm{q}_{L_2(\Gamma)}^2} \leq \overline{\lambda}~\delta^d.
\end{equation}
Thus, the condition number of the Schur complement matrix $S_\Gamma := S^{(1)} +  S^{(2)}$ has the following bound:
\begin{equation*} \label{condNumberS}
\kappa(S_\Gamma) \lesssim \delta^{-2}.
\end{equation*}
\end{theorem}

\begin{rem}
The preceding condition number estimate indicates that the condition
number of the Schur complement is no greater than that of the
corresponding stiffness matrix; see \eqref{condNumberA}. This estimate
is not tight.  In fact, we numerically observe smaller condition
numbers for the Schur complement; see Table \ref{table:suu}.
\end{rem}

\subsubsection{The Nonlocal Schur Complement Matrix}
When the contributions from each subdomain are accounted separately,
we can write $K_{\Gamma \Gamma}$ in \eqref{stiffnessMatrix} as
$K_{\Gamma \Gamma} = K_{\Gamma \Gamma}^{(1)} + K_{\Gamma \Gamma}^{(2)}$.
Then,  $S^{(i)}$ in \eqref{SchurComplement} can be written as follows:
\begin{equation*}
S^{(i)} := K_{\Gamma \Gamma}^{(i)} - K_{\Gamma i} K_{ii}^{-1} K_{i \Gamma}.
\end{equation*}
The solution across the whole of $\Gamma$ is determined by solving $S_\Gamma u_\Gamma  = \tilde{f}$ for $u_\Gamma$, where
\begin{align*}
\tilde{f} & :=  f_{\Gamma} - K_{\Gamma 1} K_{1 1}^{-1} f_{1} - K_{\Gamma 2} K_{2 2}^{-1} f_{2}.
\end{align*}
We observed in \S\ref{sec:auu_num_expts} that the condition number of the stiffness matrix $K$ depends only weakly upon the mesh size $h$. Therefore, we expect that the condition number of the Schur complement matrix $S_\Gamma$ should at most depend only weakly upon $h$. We will examine this conjecture in \S\ref{sec:suu_num_expts}.

\begin{table}[ht]
\centering
\subtable[Fixed $\delta$, vary $h$.]{
\scriptsize{
\begin{tabular}{|r|c|c|c|c|c|c|c|}
  \hline
                              &                                  & \multicolumn{3}{|c|}{Piecewise Constant Shape Functions} & \multicolumn{3}{|c|}{Piecewise Linear Shape Functions} \\
  \multicolumn{1}{|c|}{$1/h$} & \multicolumn{1}{|c|}{$1/\delta$} & \multicolumn{1}{|c|}{$\lambda_{\min}$} & \multicolumn{1}{|c|}{$\lambda_{\max}$} & \multicolumn{1}{|c|}{Condition \#} & \multicolumn{1}{|c|}{$\lambda_{\min}$} & \multicolumn{1}{|c|}{$\lambda_{\max}$} & \multicolumn{1}{|c|}{Condition \#} \\
  \hline
     2000 & 20 & 1.64E-06 & 5.01E-05 & 3.06E+01 & 1.63E-06 & 4.97E-05 & 3.04E+01 \\
     4000 & 20 & 8.21E-07 & 2.50E-05 & 3.05E+01 & 8.21E-07 & 2.49E-05 & 3.03E+01 \\
     8000 & 20 & 4.12E-07 & 1.25E-05 & 3.04E+01 & 4.12E-07 & 1.25E-05 & 3.03E+01 \\
  \hline
\end{tabular}
}
\label{table:suu:h}
}
\subtable[Fixed $h$, vary $\delta$.]{
\scriptsize{
\begin{tabular}{|r|c|c|c|c|c|c|c|}
  \hline
                              &                                  & \multicolumn{3}{|c|}{Piecewise Constant Shape Functions} & \multicolumn{3}{|c|}{Piecewise Linear Shape Functions} \\
  \multicolumn{1}{|c|}{$1/h$} & \multicolumn{1}{|c|}{$1/\delta$} & \multicolumn{1}{|c|}{$\lambda_{\min}$} & \multicolumn{1}{|c|}{$\lambda_{\max}$} & \multicolumn{1}{|c|}{Condition \#} & \multicolumn{1}{|c|}{$\lambda_{\min}$} & \multicolumn{1}{|c|}{$\lambda_{\max}$} & \multicolumn{1}{|c|}{Condition \#} \\
  \hline
     8000 & 20 & 4.12E-07 & 1.25E-05 & 3.04E+01 & 4.12E-07 & 1.25E-05 & 3.03E+01 \\
     8000 & 40 & 1.03E-07 & 6.26E-06 & 6.07E+01 & 1.03E-07 & 6.23E-06 & 6.04E+01 \\
     8000 & 80 & 2.57E-08 & 3.13E-06 & 1.22E+02 & 2.57E-08 & 3.11E-06 & 1.21E+02 \\
  \hline
\end{tabular}
}
\label{table:suu:delta}
}
\caption{Condition number for $S_\Gamma$ in 1D for (a) fixed $\delta$, allowing $h$ to vary, and (b) fixed $h$, allowing $\delta$ to vary. This data is plotted in Figures \ref{fig:suu}.} \label{table:suu}
\end{table}

\begin{figure}[ht!]
\centering
\subfigure[Fixed $\delta$, vary $h$. ]{
 \scalebox{0.42}{\includegraphics{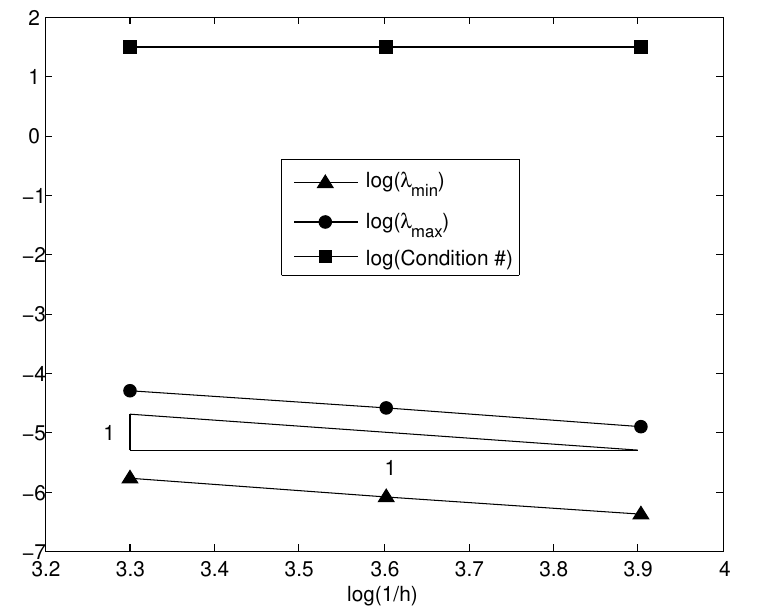}}
 \label{fig:suu:h}
}
\subfigure[Fixed $h$, vary $\delta$.]{
 \scalebox{0.42}{\includegraphics{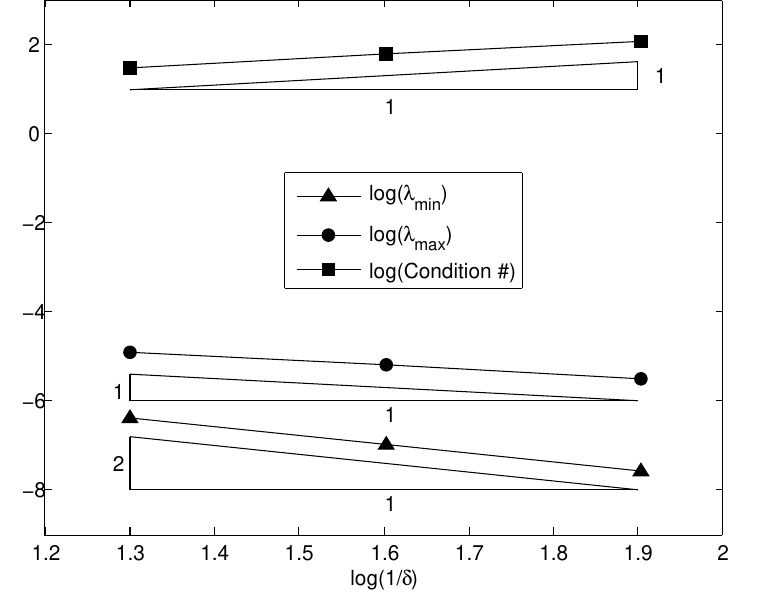}}
 \label{fig:suu:delta}
}
\caption{Condition number for $S_\Gamma$ in 1D for (a) fixed $\delta$, allowing $h$ to vary, and (b) fixed $h$, allowing $\delta$ to vary. The condition number of $S_\Gamma$ is only weakly $h$-dependent, but varies with $\delta^{-1}$. These figures are plotted from data in Table \ref{table:suu}. The plots for piecewise linear and piecewise constant shape functions are identical.} \label{fig:suu}
\end{figure}

\begin{table}[ht]
\centering
\subtable[Fixed $\delta$, vary $h$.]{
\scriptsize{
\begin{tabular}{|r|c|c|c|c|}
  \hline
  \multicolumn{1}{|c|}{$1/h$} & \multicolumn{1}{|c|}{$1 / \delta$} & \multicolumn{1}{|c|}{$\lambda_{\min}$} & \multicolumn{1}{|c|}{$\lambda_{\max}$} & \multicolumn{1}{|c|}{Condition \#} \\
  \hline
    50 & 10 & 1.14E-06 & 1.38E-05 & 1.21E+01 \\
   100 & 10 & 2.57E-07 & 3.48E-06 & 1.36E+01 \\
   200 & 10 & 6.61E-08 & 8.70E-07 & 1.32E+01 \\
  \hline
\end{tabular}
}
\label{table:suu2d:h}
}
\subtable[Fixed $h$, vary $\delta$.]{
\scriptsize{
\begin{tabular}{|r|c|c|c|c|c|}
  \hline
  \multicolumn{1}{|c|}{$1/h$} & \multicolumn{1}{|c|}{$1 / \delta$} & \multicolumn{1}{|c|}{$\lambda_{\min}$} & \multicolumn{1}{|c|}{$\lambda_{\max}$} & \multicolumn{1}{|c|}{Condition \#} \\
  \hline
   200 & 10 & 6.61E-08 & 8.70E-07 & 1.32E+01 \\
   200 & 20 & 7.87E-09 & 2.18E-07 & 2.77E+01 \\
   200 & 40 & 1.09E-09 & 4.51E-08 & 4.96E+01 \\
  \hline
\end{tabular}
}
\label{table:suu2d:delta}
}
\caption{Condition number for $S_\Gamma$ in 2D for (a) fixed $\delta$, allowing $h$ to vary, and (b) fixed $h$, allowing $\delta$ to vary. This data is plotted in Figure \ref{fig:suu2d}.} \label{table:suu2d}
\end{table}

\begin{figure}[t]
\centering
\subfigure[Fixed $\delta$, vary $h$. ]{
 \scalebox{0.42}{\includegraphics{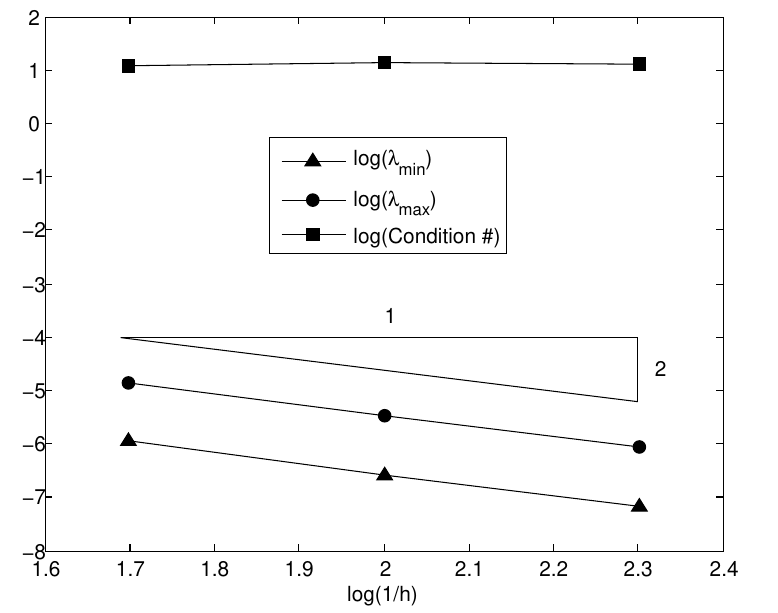}}
 \label{fig:suu2d:h}
}
\subfigure[Fixed $h$, vary $\delta$.]{
 \scalebox{0.42}{\includegraphics{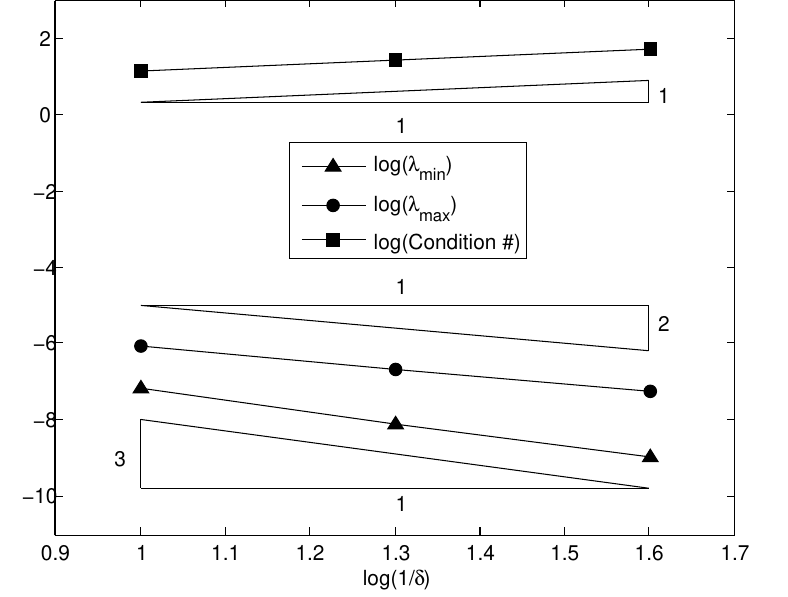}}
 \label{fig:suu2d:delta}
}
\caption{Condition number for $S_\Gamma$ in 2D for (a) fixed $\delta$, allowing $h$ to vary, and (b) fixed $h$, allowing $\delta$ to vary. The condition number of $S_\Gamma$ in 2D is only weakly $h$-dependent, but varies with $\delta^{-1}$. These figures are plotted from data in Table \ref{table:suu2d}.}
\label{fig:suu2d}
\end{figure}

\subsection{Numerical Verification of the Schur Complement Conditioning}
\label{sec:suu_num_expts}

To test the conjecture of the previous section, we discretize the Dirichlet boundary value problem
\begin{align}
 s_i(u_h,v_h) = (b,v_h) \qquad \forall v_h \in \Lambda_h,
\end{align}
with $u_h=0$ on $\mathcal{B}\Omega$, using piecewise constant and piecewise linear shape functions on uniform cartesian mesh, and numerically determine the ratio of the largest and smallest eigenvalues, defining the condition number of the problem.

\subsubsection{Results in One Dimension}

We define the regions $\Omega_1 = (0,0.5-\delta/2)$, $\Omega_2 = (0.5+\delta/2,1)$, and $\Gamma = (0.5-\delta/2,0.5+\delta/2)$, such that $\Gamma$ is always a region of width $\delta$ centered at $x=0.5$. We then compute the largest and smallest eigenvalues of $S_\Gamma$. We show results for both piecewise constant and piecewise linear shape functions to verify that the choice of shape function does not play a role in the conditioning of the discrete system.

We first compute the condition number of $S_\Gamma$ for different $h$ while holding $\delta$ fixed. Our results appear in Tables \ref{table:suu} and Figures \ref{fig:suu}. The minimum and maximum eigenvalues depend linearly on $h$, with a slope of nearly unity. Consequently, the condition number of $S_\Gamma$ is only weakly $h$-dependent. We then compute the condition number of $S_\Gamma$ for different $\delta$ while holding $h$ fixed, and observe that the condition number varies nearly as $\delta^{-1}$, which is better conditioned than the original stiffness matrix $K$, whose condition number varied with $\delta^{-2}$.

\subsubsection{Results in Two Dimensions}

We define the regions $\Omega_1 = (0,0.5-\delta/2)\times(0,1)$, $\Omega_2 = (0.5+\delta/2,1)\times(0,1)$, and $\Gamma = (0.5-\delta/2,0.5+\delta/2)\times(0,1)$, such that $\Gamma$ is always a region of width $\delta$ centered at $x=0.5$. We then compute the largest and smallest eigenvalues of $S_\Gamma$. We consider only piecewise constant shape functions in 2D, having established that the choice of shape function does not affect the conditioning.

We first compute the condition number of $S_\Gamma$ for different $h$ while holding $\delta$ fixed, and observe that minimum and maximum eigenvalues depend linearly on $h$ with a slope of approximately two, and again the condition number of $K$ depends only weakly upon the mesh size. Our results appear in Tables \ref{table:suu2d} and Figures \ref{fig:suu2d}. We then compute the condition number of $S_\Gamma$ for different $\delta$ while holding $h$ fixed, and observe that the condition number again varies as $\delta^{-1}$.


\section{Conclusions and Future Work} \label{sec:conclusions}

\begin{table}[ht]
{\footnotesize
$$
\begin{array}{ccccccc}
\text{{\bf Dim}} & \lambda_{\min}(K) & \lambda_{\max}(K) & \kappa(K) & \lambda_{\min}(S_\Gamma) & \lambda_{\max}(S_\Gamma) & \kappa(S_\Gamma)\\[1ex] \hline \hline \\
\bf{1D} & \mathcal{O}(\delta^{3}) & \mathcal{O}(\delta) & \mathcal{O}(\delta^{-2}) & \mathcal{O}(\delta^{2}) & \mathcal{O}(\delta) & \mathcal{O}(\delta^{-1}) \\[1ex]
\bf{2D} & \mathcal{O}\left(\delta^{4}\right) & \mathcal{O}(\delta^{2})  & \mathcal{O}(\delta^{-2}) & \mathcal{O}(\delta^{3}) & \mathcal{O}(\delta^{2}) & \mathcal{O}(\delta^{-1}) \\[1ex]
\hline
\end{array}
$$
\caption{The $\delta$-quantification of the reported numerical
results.} \label{table:collection}
}
\end{table}

We have presented a variational theory for nonlocal problems, such as \eqref{eq:1DomainPDstrong}. With this theory, we proved the well-posedness of the variational formulation of nonlocal boundary value problems with Dirichlet boundary conditions and practical kernel functions that are relevant to peridynamics.  In addition, we proved a spectral equivalence estimate which leads to a mesh-size independent upper bound for the condition number of the stiffness matrix. The spectral equivalence relies on the upper bound \eqref{eq:upperBound} and the nonlocal Poincar\'{e} inequality \eqref{poincareDelta} for the lower bound, where in both the $\delta$-dependence and dimension dependence have been explicitly quantified. Supporting numerical experiments demonstrated the sharpness of the upper bound \eqref{eq:upperBound} as well as the lower bound \eqref{poincareDelta}. We then constructed a nonlocal domain decomposition framework with associated nonlocal transmission conditions, also proving equivalence between the one-domain and two-domain nonlocal Dirichlet boundary value problems. We defined an energy minimizing extension, analogous to a harmonic extension used in the local case, to analyze the condition number of the nonlocal Schur complement operator. We discretized our two-domain weak form to arrive at a nonlocal Schur complement matrix. Conditioning of the nonlocal Schur complement matrix was explored via numerical studies. We summarize the numerical results in Table \ref{table:collection}. We observe that $\kappa(K)$ and $\kappa(S_\Gamma)$ are only weakly dependent upon the mesh size but vary with $\delta^{-2}$ and $\delta^{-1}$, respectively.

It is interesting to compare the conditioning of the discrete nonlocal problem with the conditioning of the (local) discrete Laplace equation. The condition number of the stiffness matrix for the local discrete Laplace equation varies with $h^{-2}$ \cite[Theorem B.32]{Toselli:2005:DDBook}, and the corresponding Schur complement matrix condition number varies with $h^{-1}$~\cite[Lemma 4.11]{Toselli:2005:DDBook}. For a fixed mesh size $0 < h \ll \delta$, we see from Table \ref{table:collection} that the discrete nonlocal stiffness matrix $K$ varies with $\delta^{-2}$, and the condition number of the corresponding nonlocal Schur complement matrix $S_\Gamma$ varies as $\delta^{-1}$.

Application of an appropriate preconditioner, involving the solution of a coarse problem, reduces the condition number of the Schur complement of the weak classical (local) Laplace operator from $\mathcal{O}\left( (H h)^{-1} \right)$ to $\mathcal{O}\left( (1 + \log(H/h))^2 \right)$, where $H$ is the subdomain size~\cite[Lemma 4.11]{Toselli:2005:DDBook}, \cite[\S4.3.6]{Smith:1996:DD}. One unexplored area involves examining the role of a coarse problem in the nonlocal setting, which has not been considered here. A logical direction would be to expand other substructuring methods to a nonlocal setting, such as Neumann-Dirichlet, Neumann-Neumann, FETI-DP (the dual-primal finite element tearing and interconnecting method) \cite{Farhat:2001:FETI-DP}, or BDDC (balancing domain decomposition by constraints) \cite{Dohrmann:2003:BDDC}. Additional opportunities for future research include addressing convergence analysis for alternative domain decomposition methods not based on substructuring in a nonlocal setting. More fundamental concepts in Schwarz theory such as stable decompositions and local solvers need to be reconstructed for nonlocal problems to support convergence analysis for additive, multiplicative, and hybrid algorithms.

\section*{Acknowledgements} \label{sec:Ack}
The first author thanks Dr. Tadele Mengesha of Louisiana State University for many enlightening discussions.  The authors also acknowledge helpful discussions with Pablo Seleson of Florida State University, and also Dr. Richard Lehoucq of Sandia National Laboratories, and thank him for pointing out the reference \cite{Andreu:2009:pLaplace}.

\bibliographystyle{elsarticle-num}
\bibliography{NonlocalDD}

\end{document}